\documentclass[bj,preprint]{imsart}

\RequirePackage[OT1]{fontenc}
\RequirePackage{amsthm,amsmath}
\RequirePackage[numbers]{natbib}
\RequirePackage[colorlinks,citecolor=blue,urlcolor=blue]{hyperref}
\usepackage[dvipsnames]{xcolor}

\usepackage{amssymb}
\usepackage{amsfonts}
\usepackage{amsthm}
\usepackage{amsmath}
\usepackage{dsfont}
\usepackage{stackengine}
\usepackage{nccmath}
\usepackage{empheq} 

\usepackage{stmaryrd}

\renewcommand{\leq}{\leqslant}
\renewcommand{\geq}{\geqslant}

\newcommand{\innerl}[2]{\langle #1, #2 \rangle}
\newcommand{\innerll}[2]{\left\langle #1, #2 \right\rangle}

\newtheorem{theorem}{Theorem}[section]
\newtheorem{proposition}{Proposition}[section]
\newtheorem{corollary}{Corollary}[section]
\newtheorem{lemma}[theorem]{Lemma}

\newtheorem{remark}{Remark}
\newtheorem{assumption}{Assumption}
\usepackage{times}
\usepackage{graphicx,epstopdf} 
\usepackage[T1]{fontenc}

\usepackage{subfigure}

\DeclareMathOperator{\ess}{ess}
\DeclareMathOperator{\E}{\mathbb{E}\,}

\DeclareMathOperator{\tr}{tr}
\DeclareMathOperator{\diag}{diag}
\DeclareMathOperator{\cond}{cond}
\usepackage{color}

\startlocaldefs
\numberwithin{equation}{section}
\theoremstyle{plain}

\endlocaldefs

\begin{document}

\begin{frontmatter}
\title{Non asymptotic estimation lower bounds for LTI state space models with Cram\'er-Rao and van Trees}
\runtitle{Estimation lower bounds for LTI state space models}

\begin{aug}
\author{\fnms{Boualem } \snm{Djehiche}\thanksref{}\ead[label=e1]{boualem@kth.se,othmane@kth.se}}
\and 
\author{\fnms{Othmane } \snm{Mazhar}\thanksref{}\ead[label=e2]{othmane@kth.se}}

\runauthor{B. Djehiche and O. Mazhar.}

\affiliation{Department of Mathematics, KTH Royal Institute of Technology\thanksmark{m1}}

\address{
Department of Mathematics\\
KTH Royal Institute of Technology\\
100 44 Stockholm, Sweden.\\
\printead{e1}\\
\phantom{E-mail:\ }}



\end{aug}

\begin{abstract}
We study the estimation problem for linear time-invariant (LTI) state-space models with Gaussian excitation of an unknown covariance. We provide non asymptotic lower bounds for the expected estimation error and the mean square estimation risk of the least square estimator, and the minimax mean square estimation risk. These bounds are sharp with explicit constants when the matrix of the dynamics has no eigenvalues on the unit circle and are rate-optimal when they do. Our results extend and improve existing lower bounds to lower bounds in expectation of the mean square estimation risk and to systems with a general noise covariance. 

Instrumental to our derivation are new concentration results for rescaled sample covariances and deviation results for the corresponding multiplication processes of the covariates, a differential geometric construction of a prior on the unit operator ball of small Fisher information, and an extension of the Cram\'er-Rao and van Trees inequalities to matrix-valued estimators.
\end{abstract}


\end{frontmatter}

\section{Introduction}

\subsection{Statement of the problem}
We consider a linear time invariant state space model parametrized over the set $\Omega = \mathcal{M}_{d \times d}(\mathbb{R})$ of $d\times d$ matrices over  $\mathbb{R}$, with a Gaussian noise $B\varepsilon_i$ where $B \in \mathcal{M}_{d \times d}(\mathbb{R})$ is a square full rank unknown matrix. The model is given by
\begin{equation}\label{Eq: state space model}
    x_{i+1} = A x_i + B\varepsilon_i
\end{equation}
for some parameter matrix $A \in \Omega$, where $(\varepsilon_i)_0^{N-1}$ is a sequence of \emph{i.i.d.} multivariate normal $\mathcal{N}(0,I_d)$. For simplicity we take  $x_0 = 0$. Expending \eqref{Eq: state space model}, we obtain the following expression for the observed covariates for $i \in \llbracket 1, N \rrbracket $ 
\begin{equation*}
    x_i = \sum \limits_{k = 0}^{i-1} A^{i-1-k}B\varepsilon_k.
\end{equation*}
 
Given the observations, the least square estimator for the parameter matrix $A$ is given as\footnote{In Lemma \ref{invertiblity of the sample covariance} below, we show that the sample covariance matrix for the model \eqref{Eq: state space model} is invertible with probability one.}  
\begin{align}\label{LS}
    \hat{A}_{\text{LS}} &= \underset{A \in \mathcal{M}_{d \times d}(\mathbb{R})}{\arg\min}\sum \limits_{i= 1}^N |x_i - Ax_{i-1}|_2^2 
    = \left(\sum \limits_{i=1}^N x_{i}x_{i-1}^*\right)\left(\sum \limits_{i=1}^N x_{i-1}x_{i-1}^*\right)^{-1},
\end{align}
where $y^*$ denotes the  transpose of the vector $y$ and $|y|^2_{2} :=\sum_{i=1}^n |y_i|^2$.

In this paper we provide non asymptotic lower bounds for the estimation risks $\mathcal{E}(\hat{A}_{\text{LS}},A)$, $\mathcal{E}_2(\hat{A}_{\text{LS}},A)$  and the minimax mean square estimation risk $\mathcal{E}_2(\mathcal{C}_s)$  over the class $\mathcal{C}_s$ of matrices $A$ whose singular values are larger than $s\ge 0$,  defined as follows. For a given estimator $\hat{A}$ of $A \in \mathcal{M}_{d \times d}(\mathbb{R})$, we define the expected estimation error as 
\begin{equation*}
    \mathcal{E}(\hat{A},A) = \E(\hat{A}-A)(\hat{A}-A)^*,
\end{equation*}
and the mean square estimation risk as
\begin{equation*}
    \mathcal{E}_2(\hat{A},A) = \E(|\hat{A}-A|_{S_2}^2) = \tr\left(\mathcal{E}(\hat{A},A) \right),
\end{equation*}
where $W^*$ is the (conjugate) transpose of the matrix $W$ and $|W|_{S_2} = (\tr(W^* W))^{1/2}$.

\noindent Furthermore, given a class of parameters $\mathcal{C} \subset \mathcal{M}_{d \times d}(\mathbb{R})$ we define the worst case mean square risk over the given class as 
\begin{equation*}
    \mathcal{E}_2(\hat{A},\mathcal{C}) := \sup \limits_{A \in \mathcal{C}} \mathcal{E}_2(\hat{A},A).
\end{equation*}
Finally, we define the minimax mean square estimation risk over the class $\mathcal{C}$ by
\begin{equation*}
    \mathcal{E}_2(\mathcal{C}) := \inf \limits_{\hat{A}} \sup \limits_{A \in \mathcal{C}} \mathcal{E}_2(\hat{A},A), 
\end{equation*}
where the infimum is taken over all $(x_1,\dots,x_N)$-measurable functions.  

Before we highlight the content of the paper, we give a short review of related results in the literature and describe the main contribution of the present study. 
\subsection{Discussion of the related literature}
The literature on the estimation of LTI state space models is very rich and its complete overview falls beyond the scope of this paper. Instead, we focus here on recent results on non asymptotic estimation lower bounds for the matrix $A$, which to the best of our knowledge represent the state of the art on the topic. We also mention a few results on non asymptotic upper bounds to figure out at least qualitatively  which lower bounds are reasonable for this estimation problem.

\begin{itemize}
    \item Non asymptotic lower bounds on the estimation risk appeared first in ~\cite[Theorem ~$2.3$]{pmlr-v75-simchowitz18a}. This work provides non asymptotic high probability minimax lower bounds on the estimation problem for the matrix $A$, in the case where the matrix $B = I_d$, in term of the operator norm. These minimax lower bounds were provided for the sub-classes of the orthogonal matrices $O(d)$ parametrized as $\mathcal{C}_{\rho} =\left\{ \rho O \ \big| \ O \in O(d) \right\}$ for some $\rho \geq 0$, and they show the existence of three different possible decay rates depending on the value of $\rho$ as follows:\\
    Set $\gamma_N(\rho) := \sum \limits_{i=0}^{N-1}\rho^{2i}$. Then, for all $\hat{A}$ and $\delta \in (0\ 1/4)$, there exists $O \in O(d)$ such that with probability at least $\delta$ we have
    \begin{equation*}
        |\hat{A} - \rho O|^2_{S_\infty} \geq C\frac{d+ \log(1/\delta)}{N\gamma_N(\rho)},   
    \end{equation*}
    where $|W|_{S_{\infty}} = \underset{|x|_2\le 1}{\max}|Wx|_2$. 
    
    This bound is to be understood as a minimax bound in the sense that  for any estimator $\hat{A}$ it provides a matrix $\rho O$ for which the probability of estimation failure is greater than $\delta$. This result also shows the existence of three different decay rates at least when estimating scaled orthogonal matrices: for stable matrices ($\rho \leq 1-1/N$) where $A$ cannot be estimated faster than $\frac{C}{N}(d+ \log(1/\delta))(1-\rho^2)$, for limit stable matrices ($\rho \in (1-1/N,\ 1+1/N)$) where $A$ cannot be estimated faster than $\frac{C}{N^2}(d+ \log(1/\delta))$, and for unstable matrices ($\rho \geq 1+1/N $) where $A$ cannot be estimated faster than $C\frac{d+ \log(1/\delta)}{N\rho^{2N}}$.
    \item To extend these results beyond the class of scaled orthogonal matrices in the same setup, \cite{Jedra19} provides high probability problem specific lower bounds in the sense that the bound specifically depends on the parameter matrix $A$. For the purpose of establishing such a bound, some conditions on the estimator have to be imposed. To this end they introduce the notion of $(\epsilon,\delta)$-locally stable estimator in $A$. An estimator is said to be $(\epsilon,\delta)$-locally stable in $A$ if there is an $N$ such that for all $ n \geq N$ and all
    $D \in B(0,3\epsilon) = \{D  \in \mathcal{M}_{d\times d}(\mathbb{R}):\,\, |A-D|_{S_2} \leq 3 \epsilon \}$, $\mathbb{P}(|\hat{A}-A|_{S_2} \leq \epsilon)\geq \delta$. They define the notion of sample complexity as the infimum over such $N$s and characterise it as the smallest $N$ satisfying 
    \begin{equation*}
        \lambda_{\min}\left(\sum \limits_{i=1}^{N-1} (N-i)A^{i-1}A^{*i-1}\right) \geq \frac{1}{2\epsilon^2 }\log\left(\frac{1}{2.4\delta}\right),
    \end{equation*}
    where $\lambda_{\text{min}}(W)$ denotes the smallest eigenvalue of the symmetric matrix $W$.

    \item For the upper bound, to the best of our knowledge, the only papers which addressed the non asymptotic estimation of the matrix $A$ with stable, unstable, and limit stable parts are \cite{pmlr-v97-sarkar19a,SHIRANIFARADONBEH2018342} and their results are obtained for the operator norm. We refer to ~\cite[Theorem ~$1$]{pmlr-v97-sarkar19a} for the exact statement of these results and the conditions under which they hold. These results confirm the existence of three decay rates in the non asymptotic case if the least square estimator is used to estimate $A$. Indeed, for a stable matrix in the sense that $|A|_{S_\infty} \leq 1-1/N$, one obtains with probability at least $1 - \delta$
    \begin{equation*}
        |\hat{A}_{\text{LS}}- A|^2_{S_\infty} \leq C\frac{d\log(Nd/\delta)}{N} \quad \text{for} \quad N \geq Cd\log(Nd/\delta).
    \end{equation*}
    If the matrix has only a limit stable part meaning  that $1-1/N \leq s_{\min}(A)\leq |A|_{S_\infty} \leq 1+1/N$, where $s_{\text{min}}(A)$ is the least singular value of $A$, they obtain with probability at least $1 - \delta$
    \begin{equation*}
        |\hat{A}_{\text{LS}}- A|^2_{S_\infty} \leq C\frac{h(d)\log^2(Nd/\delta)}{N^2} \quad \text{with $h(d)$ a function of the dimension}.
    \end{equation*}
    If the matrix has only an unstable part in the sense that $|A|_{S_\infty} \geq 1+1/N$, with probability at least $1 - \delta$, the quantity $|\hat{A}_{\text{LS}}- A|^2_{S_\infty}$ decreases exponentially like $|A|^{-2N}_{S_\infty}$, while it is not clear how this decay scales with the other parameters of the problem. 
    Theorem 2 in ~\cite{pmlr-v75-simchowitz18a} on the other hand shows that in the case of matrices with more than one part we obtain the worst case behavior.
    \item The asymptotic study of the estimation problem dates back at least to the eighties where the focus was on the consistency properties of the least square solution when estimating autoregressive processes. Some of the most notable papers are \cite{10.1214/aos/1176345697,LAI1982346,LAI19831}. These results are relevant since autoregressive processes are a special case of LTI state space models with $A$ given by the corresponding companion matrix. In general these results show that the least square estimator $\hat{A}_{\text{LS}}$ converges almost surely to $A$ and a careful investigation of the proofs suggests that the stable part converges at a rate $N^{-1}$, the limit stable part converges at a rate of at least $N^{-2}$ and the unstable part converges at a rate $|A|^{-2N}_{S_\infty}$. 
\end{itemize}

\subsection{Main contribution}
From the previous discussion on the lower bound we note that even though the results of \cite{Jedra19} are valid for all matrices not just scaled orthogonal ones as opposed to \cite{pmlr-v75-simchowitz18a}, the obtained bound does not account for the dimension $d$. The minimax rate obtained by \cite{pmlr-v75-simchowitz18a} on the other hand while capturing the dimension factor, holds only for scaled orthogonal matrices which is a small subset of matrices. Furthermore, they do not indicate clearly the dependence of the minimax rate on the spectral properties of the matrix since all the eigenvalue of a scaled orthogonal matrix have the same magnitude. This information is also missing in the upper bounds obtained in \cite{pmlr-v97-sarkar19a} and in the asymptotic studies of \cite{10.1214/aos/1176345697,LAI1982346,LAI19831}. Indeed, there is no mention of the spectral properties in the rates given by \cite{pmlr-v97-sarkar19a} and the dependence on the dimension is off by a $\log(Nd)$ factor at best. The spectral properties disappear in the papers \cite{10.1214/aos/1176345697,LAI1982346,LAI19831} when taking the limit and the use of inequalities like 
\begin{equation*}
    \frac{1}{d} \min \limits_{i \in \llbracket1 , d \rrbracket} |c_i(X_n)-\hat{c}_i(X_n)|_2 \leq \lambda_{\min}\left(\sum \limits_{i=1}^{N-1} x_ix_i^* \right) \leq d \min \limits_{i \in \llbracket 1 ,d \rrbracket} |c_i(X_n)-\hat{c}_i(X_n)|_2
\end{equation*}
shown in ~\cite[Equation ~$(3.6)$]{LAI1982346} prevents from getting a clear intuition on what is the right dependence on the dimension.

All of the non asymptotic results discussed above were derived for the case $B = I_d$ which results in an identity variance for the model \eqref{Eq: state space model}. If we take $B = \sigma I_d$, we obtain for the least square estimator 
\begin{align*}
    &\E\left(|\hat{A}_{\text{LS}}-A|_{S_2}^2\right) = \E\left(\left|\left(\sum \limits_{i=1}^N x_{i}x_{i-1}^*\right)\left(\sum \limits_{i=1}^N x_{i-1}x_{i-1}^*\right)^{-1}-A\right|_{S_2}^2\right)\\
    &= \E\left(\left|\left(\sum \limits_{i=1}^N (Ax_{i-1}+\sigma \varepsilon_{i-1})x_{i-1}^*\right)\left(\sum \limits_{i=1}^N x_{i-1}x_{i-1}^*\right)^{-1}-A\right|_{S_2}^2\right)\\
    &= \E\left(\left|\sum \limits_{i=1}^N  \varepsilon_{i-1}\frac{x_{i-1}^*}{\sigma}\left(\sum \limits_{i=1}^N \frac{x_{i-1}}{\sigma}\frac{x_{i-1}^*}{\sigma}\right)^{-1}\right|_{S_2}^2\right)
\end{align*}
which is independent of the noise variance since $\frac{x_i}{\sigma} = \sum \limits_{k = 0}^{i-1} A^{i-1-k}\varepsilon_k$.

Inspired by the above discussion, the present work provides lower bounds for the three types of risk mentioned above which address the following questions:
\begin{itemize}
    \item All the non asymptotic bounds present results with high probability for the operator norm. A valid question is: Do similar decay rates hold in expectation and for the mean square estimation risk?
    \item How does the lower bound on the mean square estimation risk depend on the dimension of the model? 
    \item How does the noise variance structure affect the estimation risk lower bound for the least square estimator? Is there any minimax risk lower bound which is independent of the noise variance?
    \item Do spectral properties of the matrices $A$ and $B$ affect the mean square estimation risk lower bound?
\end{itemize}
Based on the above discussion and the mentioned literature, a sharp lower bound for the minimax mean square estimation error over the class $\mathcal{C}_s$ of $d\times d$-matrices over $\mathbb{R}$ whose least singular values are larger than $s\ge 0$ would satisfy  $\mathcal{E}_2(\mathcal{C}_s) \geq \frac{d^2(1-s^2)}{N}.$ if  $s \in [0, 1)$, $\mathcal{E}_2(\mathcal{C}_s) \geq \frac{d^2}{s^{2N}}.$ if $s > 1$, and $\mathcal{E}_2(\mathcal{C}_1) \geq \frac{2d^2}{N(N-1)}$ if $s = 1$.
This lower bound has the right dependence on the dimension since we are estimating $d^2$ parameters, is independent of the noise covariance matrix which is consistent with the case $B = \sigma I_d$,  match the asymptotic rates while removing the extra logarithmic factors that appear in the non asymptotic upper bounds, and has sharp leading multiplicative constants that are consistent with the best known bounds. These properties should extend to a sharp lower bound for the expected estimation error and the mean square estimation risk for the least square estimator.

In short, our results answer the above questions by providing sharp lower bounds in that sense. Indeed, for the expected estimation error for the least square estimator our lower bound in Theorem \ref{Main lower bound on the risk of the least square estimator} retrieves the missing dimension factor in \cite{Jedra19} and unlike ~\cite[Theorem ~$2.3$]{pmlr-v75-simchowitz18a} it applies to all matrices not just orthogonal ones. The bound on the mean square estimation risk obtained in \eqref{simple LB} is a simplification of the one obtained in Theorem \ref{Main lower bound on the risk of the least square estimator}. It shows that the least square estimator has a decay rate lower bound whose main part depends only on the operator norm and no other spectral property on the matrix $A$ and is independent on the noise covariance structure implied by $B$. The existing lower bounds so far failed to clarify this dependence since they consider only the case $B = I_d$ and provided examples for scalar systems or scaled orthogonal matrices whose eigenvalues have the same magnitude. Propositions \ref{No limit stable} and \ref{with limit stable} provide an illustrative example for diagonalizable matrices where the spectral property of having a limit stable part affects the dependence on the dimension. They also show how the estimation risk for the least square estimator depends to a lesser extent on the noise covariance properties through its condition number and on other spectral properties of the matrix $A$ such as  its closest eigenvalue to the unit circle. Finally, sharp minimax rates over the class $\mathcal{C}_s$, given in Corollary \ref{Explicit minimax}, show the existence of the three decay rates for all estimators which were previously known only for the least square estimator from the upper bounds in \cite{pmlr-v75-simchowitz18a} and in the very specific cases of lower bounds for scalar systems and scaled orthogonal matrices. These minimax lower bounds show also the existence of rates that are independent of the matrix $B$ and depend only on the least singular value lower bound and the dimension. 

More specifically, The main contributions of the paper are given in Sections 2 to 4. Theorem \ref{Main lower bound on the risk of the least square estimator} is the main result of Section 2. It takes the form of a Cram\'er-Rao type of lower bound which shows that the expected estimation error for the least square is bounded from below as follows.
\begin{align*}
    \mathcal{E}(\hat{A}_{\text{LS}},A) \geq \frac{d^2\left(1 - \epsilon \right)^2}{\left(1 + C\Delta \right)^2}\frac{BB^*}{\sum \limits_{i=1}^{N-1} (N-i)|A^{i-1}B|_{S_2}^2},
\end{align*}
where $\Delta$ is defined in \eqref{LB inf}. Propositions \ref{bound on the spectrum for the sample covariance of a state space model} and \ref{bound on the multiplication process of a state space models} provide a high probability bound on the spectrum of the sample covariance and an $L_2$ bound for the multiplication process for the covariates generated by the dynamics \eqref{Eq: state space model}, which constitute the main probabilistic inequalities needed for the proof of Theorem \ref{Main lower bound on the risk of the least square estimator}. Section 3 is devoted to making the result of Theorem \ref{Main lower bound on the risk of the least square estimator} explicit. A task that we carry for diagonalizable matrices as an illustrative example. For this example, we need to distinguish between systems with a limit stable part and systems without a limit stable part. For systems with a limit stable part, the explicit lower bound is given in Proposition \ref{with limit stable}. This bound is rate-optimal with an unknown positive leading constant. For systems without a limit stable part the explicit lower bound is given in Proposition \ref{No limit stable}. That bound is sharp with leading constant one. These propositions provide answers to the aforementioned questions for the least squares estimator. Section 4 is devoted to the minimax mean square estimation risk over the class $\mathcal{C}_{s}$ of matrices with a least singular value larger than $s\ge 0$. Indeed, it turns out from the example studied in Section 3 for diagonalizable matrices that $\mathcal{C}_{s}$ is the right class to look at to get a uniform minimax rate. The result of that section holds for all estimators and all matrices in the class. While the main result is given in Theorem \ref{minimax}, a more explicit characterization is given in Corollary \ref{Explicit minimax} which shows the existence of three decay rates for the lower bound depending on the location of the least singular value lower bound $s$ of the matrix $A$. Roughly, we have, for $s \in (0\ 1),\,\, \mathcal{E}_2(\mathcal{C}_s) \gtrsim \frac{d^2(1-s^2)}{N}$, for $s = 1, \,\,\, \mathcal{E}_2(\mathcal{C}_1) \gtrsim \frac{\log^2(d)}{N^2}$ and $s > 1,\,\,\, \mathcal{E}_2(\mathcal{C}_s) \gtrsim \frac{d^2}{s^{2N}}$. These results are sharp with a leading constant that can be made arbitrarily close to one except when $s = 1$. They  are based on a van Trees inequality and provide a more general picture of the dependence on the dimension, the covariance, and the spectral properties since they hold for all the estimators, hence answering the aforementioned questions in their most general form. The main technical contributions are an extension of the van Trees inequality to the matrix case with explicit construction of a prior on the unit operator norm ball. The construction relies on a change of coordinate argument given in Appendix B and summarized as Proposition \ref{Change of measure theorem}. Finally, in Appendix C  we collect all technical lemmas used in different parts of this work.

\subsection*{Frequently used notation}\label{notation}
Throughout this work we use the following notation. 
\begin{itemize} 
\item We use the same standard order notation for real numbers and symmetric matrices. In the case of two symmetric matrices $M_1$ and $M_2$ the order notation $M_1 \leq M_2$ is to be understood in the L\"owner order sense in which $M_1 \leq M_2$ if and only if $M_1 - M_2$ is positive semi-definite.
\item For a matrix $W$ over the complex field $\mathbb{C}$, $W^*$ is the conjugate transpose. If the matrix $W$ is over $\mathbb{R}$, $W^*$  is to be identified with the matrix transpose. For scalar complex number $z$, $z^*$ is its complex conjugate.
\item We use the notation $W := \diag(w_1,\dots,w_d)$ for a matrix in $\mathcal{M}_{d\times d}(\mathbb{C})$ whose only non zero elements are along the diagonal and for which $W_{ii} = w_i$. We recall that $W \in \mathcal{M}_{d\times d}(\mathbb{R})$ has a singular value decomposition $W= U\Sigma V$ with both matrices $U$ and $V$ belonging to $O(d)$ the set of orthogonal matrices and the matrix of singular values $\Sigma = \diag(s_1(W),\dots,s_d(W))$ with $0 \leq s_1(W) \leq s_2(W) \leq \dots \leq s_d(W)$. We denote $s_{\max}(W) = s_d$, $s_{\min}(W) = s_1$ and the condition number $\cond(W) = \frac{s_{\max}(W)}{s_{\min}(W)}$. 
\item We recall that diagonalizable matrices in $\mathcal{M}_{d\times d}(\mathbb{R})$ are those admitting a decomposition $W= S\Gamma S^{-1}$ with $S$ invertible matrix called the change of basis matrix and $\Gamma = \diag(\lambda_1,\dots,\lambda_d)$ the eigenvalues matrix. Example of diagonalizable matrices are symmetric and positive definite matrices. For positive definite matrices, $W= S\Gamma S^{-1}$ is referred to as the spectral decomposition of $W$ and all the eigenvalues $\lambda_i$ are non negative. We denote the largest eigenvalue in magnitude of a diagonalizable matrix $W$ by $\lambda_{\max}(W)$ and the smallest one by $\lambda_{\min}(W)$ and recall that $s_{\max}(W) \geq |\lambda_{\max}(W)| \geq |\lambda_{\min}(W)| \geq s_{\min}(W)$.      
\item We denote by $(\Omega,\mathcal{F})$ the underlying measurable space. If a unique probability measure is defined on it, it will be denoted by $\mathbb{P}$ and  $\E$ will be the corresponding expectation operator. If a family of probability measures indexed by an index $\beta$ are defined on the same measurable space they will be denoted by $\mathbb{P}_\beta$ for distinction and $\E_\beta$ will be the corresponding expectation operator. 
\item All norms will be distinguished by a subscript denoting the underlying normed space, except for the standard absolute value $|\cdot|$.  The usual $p$-norms of a vector $x \in \mathbb{R}^n$ are denoted by
\begin{equation*}
    |x|_{p} = \left(\sum_{i=1}^n |x_i|^p \right)^{1/p} \quad \textrm{for }1 \leq p < \infty, \qquad |x|_{\infty} = \max \limits_{i \in \llbracket  1,n \rrbracket} |x_i|.
\end{equation*}
The $p$-Schatten norm of a matrix $A$ in the Schatten space $S_p$ is denoted by 
\begin{equation*}
    |W|_{S_p} = (\tr(W^* W)^{p/2})^{1/p} \quad \textrm{for }1 \leq p < \infty, \qquad |W|_{S_\infty} = \max \limits_{x\in \mathbb{S}_2^{n-1}} |Wx|_2,  
\end{equation*}
where $\mathbb{S}_2^{n-1} = \{x \in \mathbb{R}^n\ | \ |x|_2 \leq 1\}$
is the Euclidean unit sphere in dimension $n$.
We have  $|W|_{S_\infty} = s_{\max}(W)$. 

The spaces  $L_p(\Omega,\mathcal{F},\mathbb{P}) = \{ x\colon \Omega \to \mathbb{R} \textrm{ measurable},\,\,\, \E(|x|^p) < \infty  \}$ of random variables are endowed with the $L_p$-norms
\begin{equation*}
    |x|_{L_p} = (\E(|x|^p))^{1/p} \quad \textrm{for}\,\, 1 \leq p < \infty, \qquad |x|_{L_\infty} = \ess \sup \limits_{\omega \in \Omega} |x(\omega)|,  
\end{equation*}
where $\ess \sup$ refers to the essential supremum w.r.t.  the probability measure $\mathbb{P}$.

\item Throughout the paper $C$ will denote a positive constant whose exact value is not important for the derivation of the results and which may change from one line to another. $x \lesssim y$ is a shorthand notation for the statement 'there exists a positive constant $C$ such that $x \leq Cy$' and $x \simeq y$ means that $x \lesssim y$ and $y \lesssim x$. The minimum (maximum) of two real numbers $x$ and $y$ is denoted as $\min(x,y)= x \wedge y$ ($\max(x,y)= x \vee y$).

\end{itemize}

\section{General Cram\'er-Rao lower bound} In this section we derive a lower bound in a semi-definite sense for the expected estimation error $\mathcal{E}(\hat{A}_{\text{LS}},A)$ of the least square estimator $\hat{A}_{\text{LS}}$. This is done by providing a specific Cram\'er-Rao lower bound since $\hat{A}_{\text{LS}}$ is biased. This in turn will be made more explicit under mild assumptions on the matrix $A$ in Section 3 and will justify the choice of the class of parameters $\mathcal{C}_s$ for studying the minimax estimation risk $\mathcal{E}_2(\mathcal{C}_s)$.

First, we start by justifying taking the inverse in the definition of the least square estimator.
\begin{lemma}\label{invertiblity of the sample covariance}
    If the covariates are generated according to \eqref{Eq: state space model} then the sample covariance matrix $\sum \limits_{i=1}^N x_{i-1}x_{i-1}^*$ is invertible with probability one and the least square estimator exists almost surely.
\end{lemma}
Hence without loss of generality, we place ourselves on the event where the sample covariance matrix is invertible. While the proof of this lemma  is rather elementary, we could not find a source to refer to. The lemma is needed here since unlike in the case of the non asymptotic literature discussed in the introduction where the standard is to work on a high probability event where the sample covariance is invertible, we provide bounds in expectation and we need to care about the non-invertible case. This has also been the case in the asymptotic literature, where results were proven with probability one, the standard way of dealing with it was to add the extra assumption of the almost sure existence of $N_0$ such that $\sum \limits_{i=1}^{N_0} x_{i-1}x_{i-1}^*$ is invertible \cite[Lemma ~$1$]{10.1214/aos/1176345697}. We cannot afford to do this here either but this lemma shows that $N_0$ can be taken equal to $d$ in our case.    
\begin{proof}[Proof of Lemma \ref{invertiblity of the sample covariance}]
We start with noting that since  $(\varepsilon_i)_1^{d}$ are \emph{i.i.d.} multivariate normal $\mathcal{N}(0,I_d)$ then for any (deterministic) matrices $M_1,\dots,M_d \in \mathcal{M}_{d \times d}(\mathbb{R})$, we have  
\begin{equation}\label{independence implies freeness}
    M_1 \varepsilon_1 + M_2 \varepsilon_2 \dots+M_d \varepsilon_d = 0\,\, \text{a.s.,  if and only if}\,\, M_1 = M_2\dots = M_d = 0. 
\end{equation}
Indeed, for the if part, take the covariance operator of the sum. The only if part is trivial. 

To show the Lemma it is enough to show that $[x_1,\dots,x_d]$ are independent with probability one which guarantees the invertibility of the sample covariance and the existence of the least square estimator. For this we take $d$ scalars $a_1,\dots,a_d$ such that $\sum \limits_{i=1}^d a_ix_i= 0$ and note that
\begin{align*}
    0 = \sum \limits_{i=1}^d a_ix_i = \sum \limits_{i=1}^d a_i\sum \limits_{k=0}^{i-1} A^{i-1-k}B\epsilon_k = \sum \limits_{k=0}^{d-1} (\sum \limits_{i=1}^{d-k} a_{i+k} A^{i-1}B)\epsilon_k  
\end{align*}
Using \eqref{independence implies freeness}, the last display implies that the following set of equations hold 
\begin{equation*}
    \sum \limits_{i=1}^{d-k} a_{i+k} A^{i-1}B = 0 \quad \text{for} \quad k \in \llbracket 0, d-1 \rrbracket.
\end{equation*}
Since $B$ is invertible we can remove it. Starting from the last equation $k= d-1$ we get $a_d=0$, moving to $k=d-2$ we get $a_{d-1} = 0$, and so on to get $a_1= \dots = a_d = 0$, which means that $[x_1,\dots,x_d]$ are linearly independent with probability one.  
\end{proof}
Next, we provide a closed-form formula for the joint distribution of the covariates as defined by the model \eqref{Eq: state space model}. In this case, the covariates $x_i$ given $x_{i-1}$ are generated according to a non-degenerate join Gaussian measure $\mathcal{N}(Ax_{i-1},BB^*)$. The joint density is 
\begin{align*}
    f_{x;A}(x) &= f_{x_1,x_2,\ldots, x_N ;A,B}(x) = \prod_{i= 1}^N f_{x_i\,|\,x_{i-1};A,B}(x) \\
    &= (2\pi \det(BB^*))^{-N/2} \exp\Bigg(-\frac{1}{2}\sum \limits_{i= 1}^N |(BB^*)^{-1/2}(x_i - Ax_{i-1})|_2^2 \Bigg)\\
    &= \exp\left( \innerll{\begin{bmatrix}
    (BB^*)^{-1} A\\ -\frac{1}{2} A^* (BB^*)^{-1} A
    \end{bmatrix}}{\begin{bmatrix}
     \sum \limits_{i=1}^N x_ix_{i-1}^* \\
    \sum \limits_{i=1}^N x_{i-1}x_{i-1}^*\end{bmatrix}} -\log(Z(x))\right),
\end{align*}
which is the density of a tilted exponential family with natural parameter $\eta(A)$,  natural statistics $T(x)$ and  partition function $Z(x)$ so that
\begin{equation}\label{Eq: Tilted exponential familly}
f_{x;A}(x)=\exp\left( \innerll{\eta(A)}{T(x)} -\log(Z(x))\right),
\end{equation}
where
\begin{equation}\label{Eq: Tilted exponential familly-1}\left\{\begin{array}{lll}
\eta(A):=\begin{bmatrix}
    (BB^*)^{-1}A\\ -\frac{1}{2} A^* (BB^*)^{-1} A
    \end{bmatrix}, \quad T(x):= \begin{bmatrix}
    \Gamma(x) \\
    \Sigma(x) \end{bmatrix}= \begin{bmatrix}
    \sum \limits_{i=1}^N x_{i}x_{i-1}^* \\
    \sum \limits_{i=1}^N x_{i-1}x_{i-1}^* \end{bmatrix}, \\ Z(x) := \left( 2 \pi  \det(BB^*) \right)^{N/2} \exp\left( \frac{1}{{2\sigma_\varepsilon^2}}\sum \limits_{i=1}^{N}|(BB^*)^{-1/2}x_i|_2^2\right).
\end{array}
\right.
\end{equation}


For the likelihood function \eqref{Eq: Tilted exponential familly}, we define the logarithmic sensitivity $S_x:\mathcal{M}_{d\times d}(\mathbb{R}) \to \mathcal{M}_{d\times d}(\mathbb{R})$ and the Fisher information matrix $I(A):\mathcal{M}_{d\times d}(\mathbb{R}) \to \mathcal{M}_{d\times d}(\mathbb{R})$  by 
\begin{align*}
    S_x(A) =  \nabla_A \log(f_{x;A}(x)), \quad I(A) = \E\left( S_x(A))  S_x(A)^*\right).
\end{align*} 

In the next lemma, we provide a closed-form formula for the Fisher information $I(A)$ and a simple expression for the expected correlation of the logarithmic sensitivity with the estimation error. These formulas will be needed below for deriving both the Cram\'er-Rao lower bound in the next Theorem and the van Trees inequality in section 3. 
\begin{lemma}\label{Fisher information and sensitivity}
For the tilted exponential family \eqref{Eq: Tilted exponential familly}, the Fisher information is given by
\begin{align}\label{I}
    I(A) = \left(\sum \limits_{i=1}^{N-1} (N-i)|A^{i-1}B|_{S_2}^2 \right)(BB^*)^{-1}.
\end{align}
Moreover, for the least squares estimator $\hat{A}_{\text{LS}}$, we have
\begin{align}\label{corr}
    \E\left((\hat{A}_{\text{LS}}-A) S_x(A)^*\right) = B\E\left(\left(\sum \limits_{i=1}^{N-1} \varepsilon_ix_i^*\right)\left(\sum \limits_{i=1}^{N-1} x_ix_i^*\right)^{-1}\left(\sum \limits_{i=1}^{N-1} x_i\varepsilon_i^*\right)\right)B^{-1}.
\end{align}
\end{lemma}
\begin{proof}
The log-likelihood function of the tilted exponential family \eqref{Eq: Tilted exponential familly} reads   
\begin{align*}
    \log f_{x;A}(x) = \innerll{\begin{bmatrix}
    (BB^*)^{-1}A\\ -\frac{1}{2} A^* (BB^*)^{-1} A
\end{bmatrix}}{\begin{bmatrix}
    \Gamma(x) \\
    \Sigma(x) \end{bmatrix} } -\log(Z(x)). 
\end{align*}
Therefore, the sensitivity is
\begin{equation*}\begin{array}{lll}
    S_x(A) = \nabla_A \innerll{\begin{bmatrix}
    (BB^*)^{-1}A\\ -\frac{1}{2} A^* (BB^*)^{-1} A
\end{bmatrix} }{ \begin{bmatrix}
    \Gamma(x) \\
    \Sigma(x) \end{bmatrix} }= (BB^*)^{-1}(\Gamma(x) - A\Sigma(x))\\ \qquad\qquad
    = (BB^*)^{-1}\left(\sum \limits_{i=1}^N x_{i}x_{i-1}^* - A(\sum \limits_{i=1}^N x_{i-1}x_{i-1}^*) \right)\\
     \qquad\qquad = (BB^*)^{-1}\left(\sum \limits_{i=1}^N (x_{i} -Ax_{i-1} - B\varepsilon_{i-1})x_{i-1}^*+\sum \limits_{i=1}^N (Ax_{i-1} + B\varepsilon_{i-1})x_{i-1}^* \right. \\ \left. \qquad\qquad\qquad \qquad\qquad\qquad- A(\sum \limits_{i=1}^N x_{i-1}x_{i-1}^*)\right).
    \end{array}
\end{equation*}
Finally,
\begin{equation}\label{S}
   S_x(A)=(B^*)^{-1}\left(\sum \limits_{i=1}^N \varepsilon_{i-1}x_{i-1}^* \right)+ (BB^*)^{-1}\left( \sum \limits_{i=1}^N (x_{i} -Ax_{i-1} - B\varepsilon_{i-1})x_{i-1}^* \right).
\end{equation}
The Fisher information is
\begin{align*}
    I(A) &= \E( S_x(A) S_x(A)^*)= (B^*)^{-1}\E\left(\left(\sum \limits_{i=1}^N \varepsilon_{i-1}x_{i-1}^*\right) \left(\sum \limits_{i=1}^N \varepsilon_{i-1}x_{i-1}^*\right)^* \right)  B^{-1}\\
    &= (B^*)^{-1}\E\left(\sum \limits_{i=1}^N \sum \limits_{j=1}^N \varepsilon_{i-1} \innerl{x_{i-1}}{x_{j-1}}\varepsilon_{j-1}^* \right)  B^{-1}.
\end{align*}
Upon conditioning w.r.t. the sigma algebra $\mathcal{F}_{n} = \sigma(\varepsilon_1, \dots ,\varepsilon_{n})$, we obtain
\begin{align*}
& \E\left(\sum \limits_{i=0}^{N-1} \sum \limits_{j=0}^{N-1} \varepsilon_i\innerl{x_i}{x_j}\varepsilon_j^* \right) = \E\left(\sum \limits_{i=0}^{N-1} \sum \limits_{j=0}^{N-1} \innerl{x_{i\vee j}}{x_{i \wedge j}} \E \left (\varepsilon_{i\vee j} \varepsilon_{i \wedge j}^*| \mathcal{F}_{i \vee j - 1} \right)    \right)\\
&= \E\left(\sum \limits_{i=0}^{N-1}  \innerl{x_i}{x_i}     \right) I_d = \sum \limits_{i=1}^{N-1} \E \innerl{\sum \limits_{k = 0}^{i-1} A^{i-1-k}B\varepsilon_k}{\sum \limits_{k = 0}^{i-1} A^{i-1-k}B\varepsilon_k} I_d\\
&= \sum \limits_{i=1}^{N-1} \sum \limits_{k = 0}^{i-1} \sum \limits_{ l = 0}^{i-1} \E \innerl{ A^{i-1-k}B\varepsilon_k}{ A^{i-1-l}B\varepsilon_l} I_d \\
&= \sum \limits_{i=1}^{N-1} \sum \limits_{k = 0}^{i-1} \sum \limits_{ l = 0}^{i-1}  \innerl{ \E(\varepsilon_k \varepsilon_l^*)}{ (A^{i-1-k}B)^* (A^{i-1-l}B)} I_d\\
&= \sum \limits_{i=1}^{N-1} \sum \limits_{k = 0}^{i-1} |A^{i-1-k}B|_{S_2}^2 I_d = \sum \limits_{i=1}^{N-1} (N-i)|A^{i-1}B|_{S_2}^2 I_d.
\end{align*}
This yields
\begin{align*}
I(A) &= (B^*)^{-1}\left(\sum \limits_{i=1}^{N-1} (N-i)|A^{i-1}B|_{S_2}^2 I_d \right) B^{-1}=\left(\sum \limits_{i=1}^{N-1} (N-i)|A^{i-1}B|_{S_2}^2 \right)(BB^*)^{-1}.
\end{align*}
To derive \eqref{corr}, we note that by \eqref{LS} and \eqref{S}, we have 
\begin{align*}
    &\E\left((\hat{A}_{ls}-A) S_x(A)^*\right) = \E\left((\hat{A}_{ls} - A)\sum \limits_{i=1}^N x_{i-1}\varepsilon_{i-1}^* \right)B^{-1}\\  &= \E\left(\left(\left(\sum \limits_{i=1}^N x_ix_{i-1}^*\right)\left(\sum \limits_{i=1}^N x_{i-1} x_{i-1}^*\right)^{-1} - A\right)\sum \limits_{i=1}^N x_{i-1} \varepsilon_{i-1}^*\right) B^{-1}\\
    &= \E\left(\left(\sum \limits_{i=1}^N (Ax_{i-1} + B\varepsilon_{i-1})x_{i-1}^*\left(\sum \limits_{i=1}^N x_{i-1} x_{i-1}^*\right)^{-1} - A \right)\sum \limits_{i=1}^N x_{i-1} \varepsilon_{i-1}^*\right) B^{-1} \\
    &= B\E\left(\left(\sum \limits_{i=1}^{N-1} \varepsilon_ix_i^*\right)\left(\sum \limits_{i=1}^{N-1} x_ix_i^*\right)^{-1}\sum \limits_{i=1}^{N-1} x_i\varepsilon_i^*\right)B^{-1}.
\end{align*}
\end{proof}
Introduce
\begin{equation}\label{L_A_B_Phi} \begin{array}{lll}
   \Psi := \E\left( \sum \limits_{i=1}^{N-1} x_i x_i^*\right)= \sum \limits_{k = 1}^{N-1} (N-k) A^{k-1} BB^* A^{*k-1},  \\  \quad \mathcal{L}_{A,B} := \sup_{s \in [0, 1]} | \Psi^{-1/2} (\sum_{k=0}^{N-2}  A^k  e^{j 2\pi k s} ) B|_{S_\infty}^2,
    \end{array}
\end{equation}
and set
\begin{equation}\label{Delta}
    \Delta_{A,1}(t) := (d\vee t)\mathcal{L}_{A,B} + (d\vee t)^{1/2}\mathcal{L}^{1/2}_{A,B}, \quad \Delta_{A,2} := d\mathcal{L}_{A,B}  .
\end{equation}

We now state the main theorem of this section where we provide a lower bound for the estimation error of least square estimator.
\begin{theorem} \label{Main lower bound on the risk of the least square estimator}
For all $A \in \mathcal{M}_{d\times d}(\mathbb{R})$ and $B \in \mathcal{M}_{d \times d}(\mathbb{R}) $ with $B$ non singular, the following lower bound for the estimation error holds: 
\begin{align}\label{LB}
    \mathcal{E}(\hat{A}_{\text{LS}},A) \geq \frac{d^2\left(1 - \epsilon \right)^2}{\left(1 + C\Delta \right)^2}\frac{BB^*}{\sum \limits_{i=1}^{N-1} (N-i)|A^{i-1}B|_{S_2}^2},
\end{align}
with
\begin{equation*}\label{LB inf}
    \Delta = \Delta_{A,1}\left(\log\left(\frac{ \mathcal{L}_{A,B} }{\epsilon} \right)\right).
\end{equation*}
\end{theorem}
The derivation of this lower bound is inspired by that of the classical Cram\'er-Rao lower bound using a Schur complement. Such a derivation is outlined in \citet{Ibragimov81} Chapter I, section 7, where it is done in a somewhat general setting for a class of parameters from a general Hilbert space of vectors. Besides the use of a Schur complement, the proof is different since we are working with a specific likelihood function defined on the parametric Hilbert space of matrices with a biased least square estimator, which puts the set up outside  the classical case where one could get a general bound on the larger class of unbiased estimator. 

The derivation of the bound \eqref{LB} is based on the estimates displayed in the next two propositions where we provide a high probability deviation bound for a rescaled sample covariance of the LTI state space system evolving according to \eqref{Eq: state space model} and an upper bound on the $L_2$ norm of the multiplication process of the same covariates. These results are of independent value on their own and are obtained through a generic chaining argument. Their proofs are differed to Appendix A.
\begin{proposition}\label{bound on the spectrum for the sample covariance of a state space model}
Let $x_1, \dots, x_N$ be the time shifted covariates of linear state space model \ref{Eq: state space model}. Then, for every $t \ge 1$, we obtain, with probability $1-e^{-t}$,
\begin{equation*}
|\Psi^{-1/2} \left(\sum \limits_{i=1}^{N-1} x_i x_i^*\right) \Psi^{-1/2} - I_d|_{S_\infty} \leq C\Delta_{A,1}(t).   
\end{equation*}
\end{proposition}
Previous optimal results on the concentration of dependent covariates were obtained in \cite{10.3150/20-BEJ1262} where the authors showed that time shifted covariates concentrate around their mean at a decaying rate as illustrated in ~\cite[Theorem ~$3.4$]{10.3150/20-BEJ1262} which says that with probability at least $1-e^{-t}$ we have
\begin{equation*} 
        \left| \sum \limits_{i=1}^{N-1} x_ix_i^* - \sigma_x^2I_d \right|_{S_\infty} \lesssim  \frac{d\log(d)}{N} + \sqrt{\frac{d\log(d)}{N} }  + \frac{d}{N}t + \sqrt{\frac{d}{N} } t^{1/2}.
\end{equation*}
Up to a $\log(d)$ factor these results are similar to those obtained for the independent case in ~\cite[Theorem ~$1$]{ADAMCZAK2011195}. As we shall see in Section 3 while this will be the case for stable and unstable matrices, it will not be the case if the matrix $A$ has a limit stable part.  
\begin{proposition}\label{bound on the multiplication process of a state space models}
If the covariates $(x_i)_{i=1}^N$ are generated according to a non-degenerate join Gaussian distribution with the density \eqref{Eq: Tilted exponential familly}, then we have
\begin{align*}
    \E\left|\Psi^{-1/2}\sum \limits_{i=1}^{N-1} x_i\varepsilon_i^*\right|_{S_\infty}^2  &\leq Cd\Delta_{A,2}.
\end{align*}
\end{proposition}

\begin{proof}[Proof of Theorem \ref{Main lower bound on the risk of the least square estimator}]

For any estimator $\hat{A}$ of $A$, the following matrix 
\begin{align*}
    &\begin{bmatrix}
    \hat{A}-A\\
    \nabla_A \log(f_{x;A}(x))
    \end{bmatrix}
    \begin{bmatrix}
    (\hat{A}-A)^* &
    \nabla_A \log(f_{x;A}(x))^*
    \end{bmatrix} \\
    &= \begin{bmatrix}
    (\hat{A}-A)(\hat{A}-A)^* & (\hat{A}-A)  S_x(A)^*\\
     S_x(A) (\hat{A}-A)^* &
     S_x(A)  S_x(A)^*
    \end{bmatrix}
\end{align*}
is positive semi-definite by construction. Taking the expectation and using the Schur complement formula we obtain the inequality
\begin{align*}
    &\mathcal{E}(\hat{A}_{\text{LS}},A) = \E\left((\hat{A}-A)(\hat{A}-A)^* \right) \\
    &\geq \E\left((\hat{A}-A)  S_x(A)^*\right)I(A)^{-1}
    \E\left( S_x(A) (\hat{A}-A)^*)\right).
\end{align*}
Instead of continuing with the Cram\'er-Rao derivation, we compute explicitly the terms of the last expression since we have a specific Likelihood function and estimator. The term $S_x(A)$ depends only on the likelihood function. Thus both the sensitivity $S_x(A)$ and the Fisher information $I(A)$ can be computed independently of the choice of the estimator. These terms are computed in Lemma \ref{Fisher information and sensitivity} which yields
\begin{equation*}\label{Eq: scaled self-normalized over Fisher information}
    \mathcal{E}(\hat{A}_{\text{LS}},A) \geq \frac{B\left(\E\left(\left(\sum \limits_{i=1}^{N-1} \varepsilon_ix_i^*\right)\left(\sum \limits_{i=1}^{N-1} x_ix_i^*\right)^{-1}\left(\sum \limits_{i=1}^{N-1} x_i\varepsilon_i^*\right)\right)\right)^2B^*}{\sum \limits_{i=1}^{N-1} (N-i)|A^{i-1}B|_{S_2}^2 }.
\end{equation*}
By Proposition \ref{bound on the spectrum for the sample covariance of a state space model}  we have, with probability $1-e^{-t}$, the following estimate:
\begin{equation*}
|\Psi^{-1/2} \left(\sum \limits_{i=1}^{N-1} x_ix_i^*\right) \Psi^{-1/2} - I_d|_{S_\infty} \lesssim \Delta_{A,1}.   
\end{equation*}
Thus, with probability at least $1-e^{-t}$, we have, for some positive constant $C$,
\begin{equation*}
     \sum \limits_{i=1}^{N-1} x_ix_i^* \leq ( 1 + C\Delta_{A,1})\Psi.
\end{equation*}
We then define the event $\mathcal{A}$ of probability at least $1-e^{-t}$
\begin{equation*}
    \mathcal{A} = \left\{ \frac{1}{1+ C\Delta_{A,1}} \Psi^{-1}  \leq \left(\sum \limits_{i=1}^{N-1} x_ix_i^*\right)^{-1}\right\}.
\end{equation*}
We note that
\begin{equation*}\begin{array}{lll}
    \E\left(\left(\sum \limits_{i=1}^{N-1} \varepsilon_ix_i^*\right)\left(\sum \limits_{i=1}^{N-1} x_ix_i^*\right)^{-1}\left(\sum \limits_{i=1}^{N-1} x_i\varepsilon_i^*\right)\right)\\
    \geq \E\left(\mathds{1}_{\mathcal{A}}\left(\sum \limits_{i=1}^{N-1} \varepsilon_ix_i^*\right)\left(\sum \limits_{i=1}^{N-1} x_ix_i^*\right)^{-1}\sum \limits_{i=1}^{N-1} x_i\varepsilon_i^*\right)\\
    \geq \frac{1}{1+ C\Delta_{A,1}} \E\left(\mathds{1}_{\mathcal{A}}\left(\sum \limits_{i=1}^{N-1} \varepsilon_ix_i^*\right)\Psi^{-1}\left(\sum \limits_{i=1}^{N-1} x_i\varepsilon_i^*\right)\right)\\
    = \frac{1}{1+ C\Delta_{A,1}} \left(\E\left(\left(\sum \limits_{i=1}^{N-1} \varepsilon_ix_i^*\right)\Psi^{-1}\left(\sum \limits_{i=1}^{N-1} x_i\varepsilon_i^*\right)\right) \right. \\   \left. \qquad\qquad\qquad - \E\left(\mathds{1}_{\mathcal{A}^c}\left(\sum \limits_{i=1}^{N-1} \varepsilon_ix_i^*\right)\Psi^{-1}\left(\sum \limits_{i=1}^{N-1} x_i\varepsilon_i^*\right)\right) \right)\\
    \geq \frac{1}{1+ C\Delta_{A,1}} \left(\E\left(\left(\sum \limits_{i=1}^{N-1} \varepsilon_ix_i^*\right)\Psi^{-1}\left(\sum \limits_{i=1}^{N-1} x_i\varepsilon_i^*\right)\right) \right. \\ \left. \qquad\qquad\qquad\qquad\qquad  - \mathbb{P}(\mathcal{A}^c)\E\left( \left|\Psi^{-1/2}(\sum \limits_{i=1}^{N-1} x_i\varepsilon_i^*)\right|_{S_\infty}^2 \right)I_d \right)\\
    \geq \frac{1}{1+ C\Delta_{A,1}} \left(\E\left(\left(\sum \limits_{i=1}^{N-1} \varepsilon_ix_i^*\right)\Psi^{-1}\left(\sum \limits_{i=1}^{N-1} x_i\varepsilon_i^*\right)\right) \right. \\ \left. \qquad\qquad\qquad\qquad\qquad - e^{-t}\E\left( \left|\Psi^{-1/2}(\sum \limits_{i=1}^{N-1} x_i\varepsilon_i^*)\right|_{S_\infty}^2 \right)I_d \right).
    \end{array}
\end{equation*}
We compute the first term as follows:
\begin{align*}
&\E\left(\left(\sum \limits_{i=1}^{N-1} \varepsilon_ix_i^*\right)\Psi^{-1}\left(\sum \limits_{i=1}^{N-1} x_i\varepsilon_i^*\right)\right) \\
&= \E\left(\left(\sum \limits_{i=1}^{N-1} \varepsilon_i\left(\sum \limits_{k = 0}^{i-1} A^{i-1-k}B\varepsilon_k\right)^*\right)\Psi^{-1}\left(\sum \limits_{i=1}^{N-1} \left(\sum \limits_{k = 0}^{i-1} A^{i-1-k}B\varepsilon_k \right)\varepsilon_i^*\right)\right) \\
&= \E\left(\sum \limits_{i=1}^{N-1} \varepsilon_i\E\left(\sum \limits_{k = 0}^{i-1} \varepsilon_k^* B^* A^{i-1-k*}  \Psi^{-1} A^{i-1-k}B\varepsilon_k\,\big|\, \varepsilon_i \right)\varepsilon_i^*\right)\\
&= \E\left(\sum \limits_{i=1}^{N-1} \varepsilon_i\sum \limits_{k = 0}^{i-1}  \innerl{\E ( \varepsilon_k \varepsilon_k^*| \varepsilon_i )}{B^* A^{i-1-k*}  \Psi^{-1} A^{i-1-k}B }   \varepsilon_i^*\right)\\
&= \sum \limits_{i=1}^{N-1}  \innerl{ \sum \limits_{k = 0}^{i-1} A^{i-1-k}B  B^* A^{i-1-k*} }{\Psi^{-1}} \E\left( \varepsilon_i   \varepsilon_i^* \right)\\
&=   \innerl{\sum \limits_{i=1}^{N-1} \sum \limits_{k = 0}^{i-1} A^{i-1-k}B  B^* A^{i-1-k*} } {\Psi^{-1}} I_d = \innerl{\Psi}{\Psi^{-1}} I_d = d I_d.
\end{align*}
Furthermore, by Proposition \ref{bound on the multiplication process of a state space models}, we have
\begin{align*}
    \E \left|\Psi^{-1/2}\sum \limits_{i=1}^{N-1} x_i\varepsilon_i^*\right|_{S_\infty}^2  &\leq Cd\Delta_{A,2}.
\end{align*}
altogether we obtain the desired result
\begin{equation*}
    \E\left(\left(\sum \limits_{i=1}^{N-1} \varepsilon_ix_i^*\right)\left(\sum \limits_{i=1}^{N-1} x_ix_i^*\right)^{-1}\left(\sum \limits_{i=1}^{N-1} x_i\varepsilon_i^*\right)\right) \geq \frac{d\left(1 - Ce^{-t}\Delta_{A,2} \right)}{1 + C\Delta_{A,1} }I_d
\end{equation*}
Substituting this result in \eqref{Eq: scaled self-normalized over Fisher information} we obtain 
\begin{align*}
    &\mathcal{E}(\hat{A}_{\text{LS}},A) \geq \frac{d^2\left(1 - Ce^{-t}\Delta_{A,2} \right)^2}{(1 + C\Delta_{A,1} )^2}\frac{BB^*}{\sum \limits_{i=1}^{N-1} (N-i)|A^{i-1}B|_{S_2}^2}.
\end{align*}
To ensure $Ce^{-t}\Delta_{A,2} \leq \epsilon$,  we choose $t = \log\left(\frac{C \Delta_{A,2}}{\epsilon} \right)$ for C large enough. This gives
\begin{align*}
    &\mathcal{E}(\hat{A}_{\text{LS}},A) \geq \frac{d^2\left(1 - \epsilon \right)^2}{\left(1 + C\Delta_{A,1}\left(\log\left(\frac{ \Delta_{A,2}}{\epsilon} \right)\right) \right)^2}\frac{BB^*}{\sum \limits_{i=1}^{N-1} (N-i)|A^{i-1}B|_{S_2}^2}.
\end{align*}
Plugging in the value of $\Delta_{A,2}$ gives the desired result.
\end{proof}


\section{Specific bounds for diagonalizable matrices}

Theorem \ref{Main lower bound on the risk of the least square estimator} provides the most general results when it comes to a lower bound of the mean square estimation risk incurred by the least square estimator for the estimation of a specific matrix $A$. Yet, the expression \eqref{LB} does not explicitly show the dependence of the decay rate on the spectral properties of the estimated matrix. This will be the task of the present section. This task will be done under the simplifying assumption that the matrix $A$ is diagonalizable. The non explicit parts of \eqref{LB} are the 
\begin{equation*}
    \sum \limits_{i=1}^{N-1} (N-i)|A^{i-1}B|_{S_2} \quad \text{and} \quad \mathcal{L}_{A,B} = \sup_{s \in [0,\ 1]} | \Psi^{-1/2} (\sum_{k=0}^N  A^k  e^{j 2\pi k s} ) B|_{S_\infty}^2.
\end{equation*}
To deal with the first term we introduce the function $\Phi: [0\ \infty) \to [0\ \infty)$ that will capture the three estimation rates claimed in the above literature review:
\begin{equation*}
        \Phi(a) = \begin{cases}
        \frac{N}{1-a} + \frac{1}{(1-a)^2}\quad &\text{if} \quad a \in [0, 1),\\
        \frac{N(N-1)}{2} \quad &\text{if} \quad a = 1,\\
        \frac{a^N}{(a-1)^2}\quad &\text{if} \quad a > 1.
        \end{cases}
\end{equation*}
By Lemma \ref{Estimate for double geometric series} in Appendix C, we have $\Phi(a) \geq \sum \limits_{i=0}^{N-2} (N-1-i)a^i$ on $[0\ \infty)$. Recall also that the mean square risk of the least square estimator is expressed as
\begin{equation*}
    \mathcal{E}_2(\hat{A}_{\text{LS}},A) = \E\left(|\hat{A}_{\text{LS}}-A|_{S_2}^2\right) = \tr\left(\mathcal{E}(\hat{A}_{\text{LS}},A) \right).
\end{equation*}
In view of Theorem \ref{Main lower bound on the risk of the least square estimator}, we have
\begin{align}
    \mathcal{E}_2(\hat{A}_{\text{LS}},A) &\geq \frac{d^2\left(1 - \epsilon \right)^2}{\left(1 + C\Delta \right)^2}\frac{|B|_{S_2}^2}{\sum \limits_{i=1}^{N-1} (N-i)|A^{i-1}B|_{S_2}^2} \nonumber\\
    &\geq \frac{\left(1 - \epsilon \right)^2}{\left(1 + C\Delta \right)^2}\frac{d^2}{\sum \limits_{i=1}^{N-1} (N-i)|A|_{S_\infty}^{2(i-1)}} \geq \frac{\left(1 - \epsilon \right)^2}{\left(1 + C\Delta \right)^2} \frac{d^2}{\Phi(|A|_{S_\infty}^2)}.\label{simple LB}
\end{align}
Recall that, by the Jordan matrix decomposition, a matrix $A$ can be written as 
\begin{equation}\label{Decomposition of a matrix}
    A = S\begin{bmatrix}
    A_u& & \\
    &A_s & \\
    & &A_{\text{ls}}
    \end{bmatrix}S^{-1}:= S\tilde{A}S^{-1},
\end{equation}
via a change of basis matrix $S$, $A_s \in \mathcal{M}_{s\times s}(\mathbb{C})$ is the stable part with all $|\lambda_i(A_s)| < 1$, $A_u \in \mathcal{M}_{u\times u}(\mathbb{C})$ is the unstable part with all $|\lambda_i(A_u)| > 1$, and $A_{\text{ls}} \in \mathcal{M}_{\text{ls} \times \text{ls}}(\mathbb{C})$ with all $|\lambda_i(A_{\text{ls}})| = 1$. Also define $B= S\Tilde{B}$ and note that
\begin{align*}
    \Delta = \left(d\vee \log\left(\frac{ \mathcal{L}_{A,B}}{\epsilon} \right) \right)\mathcal{L}_{A,B} + \left(d\vee \log\left(\frac{ \mathcal{L}_{A,B} }{\epsilon} \right) \right)^{1/2}\mathcal{L}^{1/2}_{A,B}
\end{align*}
is an increasing function of $\mathcal{L}_{A,B}$. Thus, to get a sharp bound in Theorem \ref{Main lower bound on the risk of the least square estimator} we need a good upper estimate for $\mathcal{L}_{A,B}$. For this purpose, we impose the following simplifying assumption.
\begin{assumption}\label{assumption}
    We assume that $A$ is diagonalizable.
\end{assumption}
The results of the next two propositions can be obtained under the more general yet less natural Assumption \ref{second assumption} stated  below. The proofs are written so that the reader can immediately see how to extend it to the general case.  
\begin{assumption}\label{second assumption}
    We assume that the matrix $S$ in the decomposition \eqref{Decomposition of a matrix} can be chosen such that $|A_u^{-1}|_{S_\infty} \leq 1$, $|A_s|_{S_\infty} \leq 1$, and $A_{\text{ls}}$ is diagonal with $A_{\text{ls}}A_{\text{ls}}^* = I_{\text{ls}}$.
\end{assumption}
\noindent It is easy to check that Assumption 1  implies Assumption 2.
\medskip
\subsection*{\underline{\bf Case 1: The matrix $A$ has a limit stable part}}
If the matrix $A$ in the linear state space model \eqref{Eq: state space model} has a limit stable part, the following proposition provides a lower bound on the estimation error of the least square estimator.
\begin{proposition}\label{with limit stable}
Suppose the matrix $A$ of a linear state space evolution given by \eqref{Eq: state space model} has eigenvalues in the limit stable part and satisfies Assumption \ref{assumption}. Set, for any $\epsilon \in (0, 1)$,
\begin{equation*}
    \Delta_\epsilon := \left(d\vee \log\left(\frac{ \cond^2(\tilde{B}) }{\epsilon (1-|A_u^{-1}|_{S_\infty}\vee |A_s|_{S_\infty})} \right) \right)\frac{\cond^2(\tilde{B})}{1-|A_u^{-1}|_{S_\infty}\vee |A_s|_{S_\infty}}.
\end{equation*}
Then, for values of $N$ such that
\begin{equation*}
    N \geq \frac{2}{1-s_{\min}^2(A_u^{-1})\vee s_{\min}^2(A_s)}, 
\end{equation*}
we obtain the following lower bound on the mean square risk for the least square estimator,
\begin{align*}
    \mathcal{E}_2(\hat{A}_{\text{LS}},A) \gtrsim (1 - \epsilon )^2 \frac{(d/\Delta_\epsilon)^2}{\Phi(|A|_{S_\infty}^2)}.
\end{align*}
\end{proposition}
We note that in the presence of a limit stable part we loose the dimension factor $d^2$ since $d/\Delta_\epsilon \leq 1$. This will not be the case when there is no limit stable part, since $\Delta_\epsilon$ can be made arbitrarily small. On the other hand, for the case with a limit stable part we do not have concentration of the sample covariance around $\Psi$ (cf. Proposition \ref{bound on the spectrum for the sample covariance of a state space model}) but only the following deviation result 
\begin{equation*}
     \sum \limits_{i=1}^{N-1} x_ix_i^* \leq C\left(d\vee \log\left(\frac{ \mathcal{L}_{A,B} }{\epsilon} \right) \right)\mathcal{L}_{A,B}\Psi
\end{equation*}
with probability $1-e^{-t}$, since $\mathcal{L}_{A,B}$ cannot be made arbitrarily small for large sample size $N$.

\begin{proof}[Proof of Proposition \ref{with limit stable}]
From Theorem \ref{Main lower bound on the risk of the least square estimator} we have obtained the simplified lower bound of equation \eqref{simple LB}
\begin{align*}
    \mathcal{E}_2(\hat{A}_{\text{LS}},A) \geq \frac{\left(1 - \epsilon \right)^2}{\left(1 + C\Delta \right)^2} \frac{d^2}{\Phi(|A|_{S_\infty}^2)},
\end{align*}
with
\begin{equation*}
    \Delta = \left(d\vee \log\left(\frac{ \mathcal{L}_{A,B} }{\epsilon} \right) \right)\mathcal{L}_{A,B} + \left(d\vee \log\left(\frac{ \mathcal{L}_{A,B} }{\epsilon} \right)\right)^{1/2}\mathcal{L}^{1/2}_{A,B}.
\end{equation*}
To obtain the result it is enough to lower bound the term
\begin{equation*}
    \frac{d^2\left(1 - \epsilon \right)^2}{(1 + C\Delta)^2}.
\end{equation*}
this requires the following upper bound on $\mathcal{L}_{A,B}$ which is a direct consequence of Lemma \ref{Upper and lower bounds for LA} in Appendix C. For all $N$ such that
\begin{equation*}
    N \geq \frac{2}{1-s_{\min}^2(A_u^{-1})\vee s_{\min}^2(A_s)}, 
\end{equation*}
we have 
\begin{equation*}\label{upper-L-A}
    \mathcal{L}_{A,B} \leq 
     \frac{9\cond^2(\tilde{B})}{1-|A_u^{-1}|_{S_\infty}\vee |A_s|_{S_\infty}}.
\end{equation*}
Since the right hand side is greater than $1$ we can upper bound $ \Delta $ by $\Delta_\epsilon$ defined as
\begin{equation*}
    \Delta_\epsilon = \left(d\vee \log\left(\frac{ \cond^2(\tilde{B}) }{\epsilon (1-|A_u^{-1}|_{S_\infty}\vee |A_s|_{S_\infty})} \right) \right)\frac{\cond^2(\tilde{B})}{1-|A_u^{-1}|_{S_\infty}\vee |A_s|_{S_\infty}} \gtrsim \Delta.
\end{equation*}
Hence, for
\begin{equation*}
    N \geq \frac{2}{1-s_{\min}^2(A_u^{-1})\vee s_{\min}^2(A_s)}, 
\end{equation*}
we obtain the following lower bound on the risk, valid for all $\epsilon \in (0, 1)$,
\begin{align*}
    \mathcal{E}_2(\hat{A}_{\text{LS}},A) \gtrsim (1 - \epsilon )^2 \frac{(d/\Delta_\epsilon)^2}{\Phi(|A|_{S_\infty}^2)}.
\end{align*}
\end{proof}

\subsection*{\underline{\bf Case 2: The matrix $A$ has no  limit stable part}} 
Now we turn in the next proposition to lower bounds on the estimation error of the least square estimator when $A$ has no limit stable part.
\begin{proposition}\label{No limit stable}
Suppose the matrix $A$ has no eigenvalues in the limit stable part and satisfies Assumption \ref{assumption}. For $\epsilon \in (0\ 1)$ and for values of $N$ such that
\begin{equation*}
    N \geq \frac{(d\vee \log(1/\epsilon))\cond(\tilde{B})^2}{\epsilon^2 (1-|A_u^{-1}|_{S_\infty}\vee |A_s|_{S_\infty})},
\end{equation*}
the expected mean square estimation error associated to the least square estimator $\hat{A}_{\text{LS}}$ satisfies the following lower bound:
\begin{align*}
    \mathcal{E}_2(\hat{A}_{\text{LS}},A) \geq \frac{(1 - \epsilon )^2}{(1 + \epsilon )^2} \frac{d^2}{\Phi(|A|_{S_\infty}^2)}.
\end{align*}
\end{proposition}
This result is sharp and the leading  multiplicative constant can be made arbitrarily close to one. Here we retrieve the dimension factor $d^2$ as opposed to the case with a limit stable part. This factor is usually expected when estimating $d^2$ parameter. The fact that it disappears in the case of presence of a limit stable part is due to the $\mathcal{L}_{A,B}$ factor which comes from the $\gamma_1$ functional in both the proof of Theorem \ref{bound on the spectrum for the sample covariance of a state space model} and \ref{bound on the multiplication process of a state space models}. One would think that this is an artifact of the proof but the same happens in the next section while deriving lower bounds on the minimax risk using completely different machinery. 
\begin{proof}
In the same way as for the case with a limit stable part, recall that by Theorem \ref{Main lower bound on the risk of the least square estimator} we have obtained the simplified lower bound of equation \eqref{simple LB}
\begin{align*}
    \mathcal{E}_2(\hat{A}_{\text{LS}},A) \geq \frac{\left(1 - \epsilon \right)^2}{\left(1 + C\Delta \right)^2} \frac{d^2}{\Phi(|A|_{S_\infty}^2)},
\end{align*}
with the definition of
\begin{equation*}
    \Delta = \left(d\vee \log\left(\frac{ \mathcal{L}_{A,B} }{\epsilon} \right) \right)\mathcal{L}_{A,B} + \left(d\vee \log\left(\frac{ \mathcal{L}_{A,B} }{\epsilon} \right)\right)^{1/2}\mathcal{L}^{1/2}_{A,B}.
\end{equation*}
Since the matrix $A$ has no limit stable part, as direct consequence of Lemma \ref{Upper and lower bounds for LA} in Appendix C, we have the bound
\begin{equation}\label{L-A-upper-2}
    \mathcal{L}_{A,B} \leq 
     \frac{4\cond(\tilde{B})^2(1-s_{\min}^2(A_s) \wedge s_{\min}^2(A_u^{-1}))}{N(1-|A_u^{-1}|_{S_\infty}\vee |A_s|_{S_\infty})}.
\end{equation}
for
\begin{equation*}
    N \geq \frac{2}{1-s_{\min}^2(A_u^{-1})\vee s_{\min}^2(A_s)},
\end{equation*}
If we take  
\begin{equation*}
    N \geq 16\frac{1}{1-s_{\min}^2(A_u^{-1})\vee s_{\min}^2(A_s)} \vee \frac{d\vee \log(1/\epsilon)}{\epsilon^2}\frac{\cond(\tilde{B})^2(1-s_{\min}^2(A_s) \wedge s_{\min}^2(A_u^{-1}))}{1-|A_u^{-1}|_{S_\infty}\vee |A_s|_{S_\infty}},
\end{equation*}
we get in particular $4 d \mathcal{L}_{A,B} \leq \epsilon^2$ and $\log \left(\frac{\mathcal{L}_{A,B}}{\epsilon}\right)\leq \log(\epsilon) \leq 0$ which means
\begin{equation*}
    \Delta \leq 2\left( d \mathcal{L}_{A,B}\right)^{1/2} \leq \epsilon.
\end{equation*}
This indeed ensures that the following lower bound on the risk holds for all $\epsilon \in (0, 1)$,
\begin{align*}
    \mathcal{E}_2(\hat{A}_{\text{LS}},A) \gtrsim \frac{(1 - \epsilon )^2}{(1 + \epsilon )^2} \frac{d^2}{\Phi(|A|_{S_\infty}^2)}.
\end{align*}
The condition on the values of $N$ is simplified since 
\begin{equation*}
    \frac{1}{1-s_{\min}^2(A_u^{-1})\vee s_{\min}^2(A_s)} \leq \frac{1}{1-|A_u^{-1}|_{S_\infty}\vee |A_s|_{S_\infty}}.
\end{equation*}
\end{proof}
\begin{remark}
The rates provided in Propositions \ref{with limit stable} and \ref{No limit stable} seem to suggest that there is a sharp phase transition from a decay rate of $N^{-1}$ in the stable case, to $N^{-2}$ in the limit stable case, to $|A|_{S_\infty}^{-2N}$ in the unstable case expressed in terms of the function $\Phi(|A|_{S_\infty}^2)$. According the these propositions the values of validity of these rates for $N$ get arbitrarily large if $|A_u^{-1}|_{S_\infty}\vee |A_s|_{S_\infty}$ is arbitrary close to $1$. This is just an artifact of the proof that has nothing to do with the lower bound provided by Theorem \ref{LB} and the transition from one rate to another is continuous in the sample size $N$. Of course, these issues disappear in asymptotic results after taking the limit such as in \cite{10.1214/aos/1176350711}. A similar remark is made for the minimax rate in Corollary \ref{Explicit minimax}. We propose a way to deal with this issue in Remark \ref{Important remak}, below. 
\end{remark}
The results presented in this section provide answers to the questions mentioned in the introduction. Indeed \eqref{simple LB} as a direct consequence of Theorem \ref{LB}, shows that we can recover similar results in expectation for the mean square error of the least square estimator in the setup of \eqref{Eq: state space model}. The observed decay rate depends on the spectrum of the matrix $A$ and is dictated by its operator norm. Proposition \ref{No limit stable} shows that in the absences of a limit stable part the decay rate is independent of the covariance and scales with $d^2$, while in the presence of a limit stable part Proposition \ref{with limit stable} shows that the decay rate still involves the noise covariance through the condition number of $B$ while the $d^2$ factor is lost. \\
These results are extended in the next section to the mean square minimax setup which holds for all estimators and larger classes of estimates. There the least singular value dictates the rate instead which is always independent of the covariance structure and the $d^2$ factor still disappears for limits stable matrices.

\section{Mini-max lower-bounds and adaptivity: The van Trees inequality}
In the previous sections, we provided lower bounds on the mean square estimation risk and the expected estimation error that are specific to the least square estimator $\hat{A}_{\text{LS}}$ when applied to a specific matrix $A$. We now turn to the more general task of describing the minimax mean square estimation risk for all estimators and over a large class of matrices. Upon inspection of \eqref{simple LB}, we note that the most difficult matrices to estimate in any class would be those with the smallest $|A|_{S_\infty}$. This remark motivates the choice of the following class $\mathcal{C}_s$ of matrices parametrized by $s \geq 0$ and defined by
\begin{equation}\label{C-s}
    \mathcal{C}_s = \left\{ A \in \mathcal{M}_{d\times d}(\mathbb{R}),\,\, \ s_{\min}(A) \geq s\right\}.
\end{equation}
So this section is devoted to providing lower bounds for the minimax mean square estimation risk 
\begin{equation*}
    \mathcal{E}_2(\mathcal{C}_s) := \inf \limits_{\hat{A}} \sup \limits_{A \in \mathcal{C}_s} \mathcal{E}_2(\hat{A},A), 
\end{equation*}
where the infimum is taken over all $(x_1,\ldots,x_N)$-measurable functions.

\medskip
The main result of this section is 
\begin{theorem}\label{minimax}
for all $s \geq 0$ and $\varepsilon > 0$ we have
\begin{align*}
       \mathcal{E}_2(\mathcal{C}_s) \geq \frac{d^2}{\sum \limits_{i=0}^{N-2} (N-1-i)(s+\epsilon)^{2i}  + \frac{2(d+2)^2}{\epsilon^2}} .
\end{align*}
\end{theorem}
This theorem is actually a van Trees inequality for the mean square risk of matrices. One might think that it is possible to use existing multivariate van Trees inequalities such as in \cite{Gill,Stoltz13} to obtain the result. But a close inspection of the proof reveals that a key step in the derivation is to construct a prior that is at the same time easy to differentiate with respect to $A$ as a matrix and easy to integrate with respect to its singular values and singular vectors. This is not an easy task if we use existing multivariate van Trees inequalities, since they are defined in terms of vector-valued gradients which are evaluated at vectors and use integrals with respect to the Lebesgue measure over the components of those vectors.    
\begin{proof}
Given $s \geq 0$ and $\varepsilon > 0$, consider the sets
\begin{equation*}
    B_\infty\left(s, \epsilon \right) = \left\{A \in \mathcal{M}_{d\times d}(\mathbb{R}):\,\,\, \left|A-\left( s + \frac{\epsilon}{2} \right)I_d\right|_{S_\infty} \leq \frac{\epsilon}{2} \right\}.
\end{equation*}
These are sets of invertible matrices with the following property: 
\begin{equation*}
    s \leq s_1(A) \leq s_2(A) \leq \dots \leq s_d(A) \leq s + \epsilon. 
\end{equation*}
Thus, for all $s \geq 0$, they are included in the class $\mathcal{C}_s$. 

For any estimator $\hat{A}$ and any prior with density $\Pi_{s,\epsilon}$ on the space of matrices $\mathcal{C}_s$, we define
\begin{equation*}
I_{\Pi_{s,\epsilon}}:= \E_{\Pi_{s,\epsilon}} \left(\nabla_A \log({\Pi_{s,\epsilon}}(A))) \nabla_A \log({\Pi_{s,\epsilon}}(A))^*\right) \end{equation*} and
\begin{equation*}
    \quad  I(A) := \E_A\left(\nabla_A \log(f_{x;A}(x)) \nabla_A \log(f_{x;A}(x))^*\right).
\end{equation*}
To obtain a lower bound on $\mathcal{E}_2(\mathcal{C}_s)$ we note that the following matrix is positive semi-definite by construction:
{\small \begin{equation*}\begin{array}{lll}
    \E_{\Pi_{s,\epsilon}} \E_A \begin{bmatrix}
    (\hat{A}-A)\\
    \nabla_A \log(f_{x;A}(x){\Pi_{s,\epsilon}}(A))
    \end{bmatrix}
    \begin{bmatrix}
    (\hat{A}-A)^* &
    \nabla_A \log(f_{x;A}(x){\Pi_{s,\epsilon}}(A))^*
    \end{bmatrix} \\
    = \E_{\Pi_{s,\epsilon}} \E_A\begin{bmatrix}
    (\hat{A}-A)(\hat{A}-A)^* & (\hat{A}-A) \nabla_A \log(f_{x;A}(x){\Pi_{s,\epsilon}}(A))^*\\
    \nabla_A \log(f_{x;A}(x){\Pi_{s,\epsilon}}(A)) (\hat{A}-A)^* & \quad
    \nabla_A \log(f_{x;A}(x){\Pi_{s,\epsilon}}(A)) \nabla_A \log(f_{x;A}(x){\Pi_{s,\epsilon}}(A))^*
    \end{bmatrix}.
    \end{array}
\end{equation*}
}
 The Schur complement formula yields the inequality
\begin{equation}\label{C-s-1}\begin{array}{lll}
    \mathcal{E}_2(\mathcal{C}_s) = \inf \limits_{\hat{A}} \sup \limits_{A \in \mathcal{C}_s} \mathcal{E}_2(\hat{A},A) \geq  \inf \limits_{\hat{A}} \tr \left( \E_{\Pi_{s,\epsilon}} \E_A\left((\hat{A}-A)(\hat{A}-A)^* \right)\right) \\ \quad 
    \geq \inf \limits_{\hat{A}} \tr \left(\E_{\Pi_{s,\epsilon}} \E_A \left((\hat{A}-A) \nabla_A \log(f_{x;A}(x){\Pi_{s,\epsilon}}(A))^*\right)I_F^{-1} \right. \\ \left. \qquad\qquad\qquad
    \E_A\left(\nabla_A \log(f_{x;A}(x){\Pi_{s,\epsilon}}(A)) (\hat{A}-A)^*)\right) \right)
    \end{array}
\end{equation}
with 
\begin{equation*}\begin{array}{lll}
    I_F = \E_{\Pi_{s,\epsilon}} \E_A \left( \nabla_A \log(f_{x;A}(x){\Pi_{s,\epsilon}}(A)) \nabla_A \log(f_{x;A}(x){\Pi_{s,\epsilon}}(A))^* \right) \\ \quad\,\,
    = \E_{\Pi_{s,\epsilon}} \E_A \left( (\nabla_A \log(f_{x;A}(x)) + \nabla_A \log({\Pi_{s,\epsilon}}(A))) \right. \\ \left.  \qquad\qquad\qquad\qquad (\nabla_A \log(f_{x;A}(x)) + \nabla_A\log({\Pi_{s,\epsilon}}(A))))^* \right) 
    \end{array}
    \end{equation*}
which becomes
\begin{equation}\label{Eq: last two terms in the main derivation of the van Trees inequality}\begin{array}{lll}
    I_F= \E_{\Pi_{s,\epsilon}}(I_{(x)}) + I_{\Pi_{s,\epsilon}} + \E_{\Pi_{s,\epsilon}}( \E_A ( \nabla_A \log(f_{x;A}(x)) \nabla_A\log({\Pi_{s,\epsilon}}(A))^*)\\
    \qquad \qquad \qquad \qquad  +\E_{\Pi_{s,\epsilon}} ( \nabla_A \log({\Pi_{s,\epsilon}}(A)) \E_A (\nabla_A \log(f_{x;A}(x)^* )). 
\end{array}
\end{equation}
By the definition of $f_{x;A}(x)$ and the independence between $\varepsilon_i$ and $x_i$ for each $i$, we have
\begin{equation*}
    \E_A (\nabla_A \log(f_{x;A}(x)) = \E_A (\sum \limits_{i=1}^N x_i\varepsilon_i^*) = \sum \limits_{i=1}^N \E_A ( x_i)\E(\varepsilon_i^*) = 0
\end{equation*}
so, the last two terms of equation \eqref{Eq: last two terms in the main derivation of the van Trees inequality} are equal to $0$ and we obtain
\begin{equation*}
    I_F = \E_{\Pi_{s,\epsilon}}(I_{(x)}) + I_{\Pi_{s,\epsilon}}.
\end{equation*}
The first term of the r.h.s of the second inequality in \eqref{C-s-1} simplifies to 
\begin{align*}
    &\E_{\Pi_{s,\epsilon}} \E_A \left((\hat{A}-A) \nabla_A \log(f_{x;A}(x){\Pi_{s,\epsilon}}(A))^*\right) \\
    &= \E_{\Pi_{s,\epsilon}} \E_A \left((\hat{A}-A) (\nabla_A \log(f_{x;A}(x)) +\nabla_A \log({\Pi_{s,\epsilon}}(A)))^*\right)\\
    &= \E_{\Pi_{s,\epsilon}} \E_A \left(\hat{A}\nabla_A \log(f_{x;A}(x))^*\right) + \E_{\Pi_{s,\epsilon}} \E_A \left(\hat{A} \nabla_A \log({\Pi_{s,\epsilon}}(A))^*\right) \\
    &\qquad \qquad  - \E_{\Pi_{s,\epsilon}} \E_A \left(A \nabla_A \log(f_{x;A}(x))^*\right) - \E_{\Pi_{s,\epsilon}} \E_A \left(A \nabla_A \log({\Pi_{s,\epsilon}}(A))^*\right)\\
    &= \E_{\Pi_{s,\epsilon}} \E_A \left(\hat{A}\nabla_A \log(f_{x;A}(x))^*\right) + \E_{\Pi_{s,\epsilon}} \E_A \left(\hat{A} \nabla_A \log({\Pi_{s,\epsilon}}(A))^*\right) \\
    &\qquad \qquad - \E_{\Pi_{s,\epsilon}} \E_A \left(A \nabla_A \log(f_{x;A}(x))^*\right) - \E_{\Pi_{s,\epsilon}} \E_A \left(A f_{x;A}(x)\nabla_A \log({\Pi_{s,\epsilon}}(A))^*\right)\\
    &= \int \left( \hat{A} {\Pi_{s,\epsilon}}(A) \nabla_A f_{x;A}(x)^* + \hat{A} f_{x;A}(x)\nabla_A {\Pi_{s,\epsilon}}(A)^* dA dx\right) \\
    &\qquad \qquad - \E_{\Pi_{s,\epsilon}} \left(A \E_A ( \nabla_A \log(f_{x;A}(x)))^*\right) - \E_{\Pi_{s,\epsilon}} \left(A \nabla_A \log({\Pi_{s,\epsilon}}(A))^*\right)\\
    &= \int \left( \hat{A} \left( \int  \nabla_A  (f_{x;A}(x){\Pi_{s,\epsilon}}(A) ) dA\right)^* dx\right) - \E_{\Pi_{s,\epsilon}} \left(A \nabla_A \log({\Pi_{s,\epsilon}}(A))^*\right).
\end{align*}
By choosing a prior $\Pi_{s,\epsilon}$ that vanishes at the boundaries of the set $B_\infty\left(s, \epsilon \right)$, the inner integral vanishes by Stoke's theorem:
\begin{equation*}
    \int  \nabla_A  (f_{x;A}(x){\Pi_{s,\epsilon}}(A) ) dA = \int_{B_\infty\left(s, \epsilon \right)}    f_{x;A}(x){\Pi_{s,\epsilon}}(A)  dA = 0.
\end{equation*}
Hence, we get
\begin{align*}
    &\E_{\Pi_{s,\epsilon}} \E_A \left((\hat{A}-A) \nabla_A \log(f_{x;A}(x){\Pi_{s,\epsilon}}(A))^*\right) = - \E_{\Pi_{s,\epsilon}} \left(A \nabla_A \log({\Pi_{s,\epsilon}}(A))^*\right).
\end{align*}
All together these considerations lead us to the following simple lower bound on the risk
\begin{equation*}\begin{array}{lll}
    \mathcal{E}_2(\mathcal{C}_s) \geq \inf \limits_{\hat{A}} \tr \left(\E_{\Pi_{s,\epsilon}} \left(A \nabla_A \log({\Pi_{s,\epsilon}}(A))^*\right)\left( \E_{\Pi_{s,\epsilon}}(I_{(x)}) + I_{\Pi_{s,\epsilon}}\right)^{-1}\right. \\  \left. \qquad\qquad\qquad\qquad \E_{\Pi_{s,\epsilon}} \left(A \nabla_A \log({\Pi_{s,\epsilon}}(A))^*\right)^* \right).
    \end{array}
\end{equation*}
We thus make the following choice of prior
\begin{equation}\label{prior}
    \Pi_{s,\epsilon}(A):= Z_{s,\epsilon} \det(\epsilon^{1/2} I_d - ((A-sI_d)(A-sI_d)^*)^{1/2})^2 \mathds{1}(A \in B_\infty\left(s, \epsilon \right)), 
\end{equation}
where $Z_{s,\epsilon} = \prod \limits_{i=1}^d Z_{i,s,\epsilon}$ is the product of integration constants to make the prior into a valid probability density. They depend on $s$ and $\epsilon$ and will be made explicit later on (see \eqref{Z-i} in Proposition \ref{Change of measure theorem}, below). Let $U\Sigma V$ be the singular value decomposition of $A- sI$ with $\Sigma$ being the diagonal matrix of the singular values $(\sigma)_{i=1}^d$ satisfying $0\leq \sigma_1 \leq \sigma_2 \leq \dots \leq \sigma_d\leq \epsilon$. Then we obtain
\begin{align} \label{rewriting of the prior}
    \Pi_{s,\epsilon}(A) &= Z_{s,\epsilon} \det(\epsilon I_d - ((A-sI_d)(A-sI_d)^*)^{1/2})^2 \mathds{1}(A \in B_\infty\left(s, \epsilon \right)) \nonumber \\
    &= Z_{s,\epsilon} \det(\epsilon I_d -  U\Sigma U^*)^2 \mathds{1}(\sigma_d \in [0 \ \epsilon]) \nonumber \\
    &= \prod \limits_{i=1}^d Z_{i,s,\epsilon}(\epsilon - \sigma_i)^2 \mathds{1}(\sigma_i \in [0 \ \epsilon]).
\end{align}
The next proposition addresses the issues of integration with respect to the representation of $\Pi_{s,\epsilon}(A)$ where $(U,V,\sigma_1,\dots,\sigma_d)$ are the integrators instead of $A$. It is a change of coordinate result whose proof is deferred to Appendix B.

\begin{proposition}\label{Change of measure theorem}
    Let $f:\mathcal{M}_{d\times d}(\mathcal{A}) \to \mathbb{R}$ be an integrable function of $dA$ the Lebesgue measure on $B_\infty\left(s, \epsilon \right)$ and, for all $i \in [1\ d]$, let $d\sigma_i$ be the Lebesgue measure on $[0\ \varepsilon]$. There exist two probability measures $dU$ and $dV$ on the orthogonal group of matrices $O(d)$ such that
{\small \begin{align*}
    \E(f(A)) = \int f(A(U,\sigma_1,\dots,\sigma_d,V))  \left( \prod \limits_{i=1}^d Z_{i,s,\epsilon}(\epsilon - \sigma_i)^2\sigma_i^{d-1}\mathds{1}(\sigma_i \in [0 \ \epsilon])d\sigma_i\right)dUdV,
\end{align*}
}
where
\begin{equation}\label{Z-i}
    Z_{i,s,\epsilon} = \left(\int \limits_0^\epsilon (\epsilon - \sigma_i)^2\sigma_i^{d-1}d\sigma_i\right)^{-1} =  \frac{d(d+1)(d+2)}{2\epsilon^{d+2}}.
\end{equation}
\end{proposition}

\begin{proof}
    See Appendix B.
\end{proof}

In view of Proposition \ref{Change of measure theorem} it is clear that if $A$ is in the boundary of $B_\infty\left(s, \epsilon \right)$, we have $\Pi_{s,\epsilon}(A) = 0$. Moreover, we can hope for independence of priors after a change of variables from $A$ to $(U,V,\sigma_1,\dots,\sigma_d)$. Hence, with this choice of prior we only need to compute $\E_{\Pi_{s,\epsilon}}(I_{(x)}) $, $ I_{\Pi_{s,\epsilon}}$, and $-\E_{\Pi_{s,\epsilon}} \left(A \nabla_A \log({\Pi_{s,\epsilon}}(A))^*\right)$ as we argued previously. We start with the term $I_{\Pi_{s,\epsilon}}$. 

To compute $-\E_{\Pi_{s,\epsilon}} \left(A \nabla_A \log({\Pi_{s,\epsilon}}(A))^*\right)$ we need a formula for computing the gradient $\nabla_A\log({\Pi_{s,\epsilon}}(A))$ in the new coordinates $(U,V,\sigma_1,\dots,\sigma_d)$. This is given in the next lemma whose proof is to the end of this section.
\begin{lemma}\label{Computing the gradient in the new coordinates} Set $\tilde{A}:= A-sI_d$.
We have
\begin{align*}
    \nabla_A \log \Pi_{s,\epsilon}(A) &= - 2 \left(\epsilon I_d- (\tilde{A}\tilde{A}^*)^{1/2} \right)^{-1} \left(\tilde{A}\tilde{A}^*\right)^{-1/2}\tilde{A}\\
    &= - 2 U\left( \epsilon I_d- \Sigma \right)^{-1} V^*.
\end{align*}
\end{lemma}
In view of Lemma \ref{Computing the gradient in the new coordinates}, we have
\begin{align*}
    &-\E_{\Pi_{s,\epsilon}} \left(A \nabla_A \log({\Pi_{s,\epsilon}}(A))^*\right) =  2\E_{\Pi_{s,\epsilon}} \left((sI_d + U\Sigma V^*)  V\left( \epsilon I_d- \Sigma \right)^{-1} U^* \right)\\
    &=  2s \int_V   V dV \int_\Sigma \Pi_{s,\epsilon}(\Sigma) \left(\epsilon I_d- \Sigma \right)^{-1} d\Sigma  \int_U U^* dU + 2\E_{\Pi_{s,\epsilon}} \left(U\Sigma \left( \epsilon I_d- \Sigma \right)^{-1} U^* \right)\\
    &=2 \int_U U \left(\int_\Sigma  \Sigma\left( \epsilon I_d- \Sigma \right)^{-1} \Pi_{s,\epsilon}(\Sigma) d\Sigma \right) U^* dU\\ & = 2 \int_{U} U U^* dU Z_{1,s,\epsilon}\Big(\int_{\Sigma} \left( \epsilon - s \right) s \mathds{1}(s \in [0 \ \epsilon]) s^{d-1}ds \Big) I_d \\ &= 2 Z_{1,s,\epsilon}\frac{\epsilon^{d+2}}{(d+1)(d+2)} I_d = dI_d. \qquad\text{(by \eqref{Z-i})}
\end{align*}
We also have
\begin{equation}\label{I-Pi}
    I_{\Pi_{s,\epsilon}}=\frac{2(d+1)(d+2)}{\epsilon^2}I_d.
\end{equation}
Indeed, using Lemma \ref{Computing the gradient in the new coordinates} we have
\begin{align*}
    I_{\Pi_{s,\epsilon}} &= \E_{\Pi_{s,\epsilon}}((\nabla_A \log \Pi_{s,\epsilon}(A))(\nabla_A \log \Pi_{s,\epsilon}(A))^*)
    = 4\E\left( U\left( \epsilon I_d- \Sigma \right)^{-2}   U^* \right) \\
    &= 4 \int_{U} U \Big(\int_{\Sigma} \left( \epsilon I_d- \Sigma \right)^{-2}  \prod \limits_{i=1}^d Z_{i,s,\epsilon}(\epsilon - s_i)^2 \mathds{1}(s_i \in [0 \ \epsilon]) s_i^{d-1}ds_i \Big) U^* dU \\
    &= 4 \int_{U} U U^* dU Z_{1,s,\epsilon}\Big(\int_{\Sigma} \left( \epsilon - s \right)^{-2} (\epsilon - s)^2 \mathds{1}(s \in [0 \ \epsilon]) s^{d-1}ds \Big) I_d \\
    &= 4 Z_{1,s,\epsilon}\Big(\int_0^\epsilon s^{d-1}ds \Big) I_d = 4 Z_{1,s,\epsilon}\frac{\epsilon^{d}}{d} I_d = \frac{2(d+1)(d+2)}{\epsilon^2}I_d. \qquad\text{(by \eqref{Z-i})}
\end{align*}
The choice of prior we made in \eqref{prior} actually defers all the difficulty of the problem into the previous two calculations. Indeed, we can get a straightforward estimate for the term $\E_{\Pi_{s,\epsilon}}(I_{(x)})$ as follows:
\begin{align*}
    \E_{\Pi_{s,\epsilon}}(I_{(x)}) &= \E_{\Pi_{s,\epsilon}}(\E_A\left(\nabla_A \log(f_{x;A}(x)) \nabla_A \log(f_{x;A}(x))^*\right))\\
    &= \E_{\Pi_{s,\epsilon}}\left(\left(\sum \limits_{i=1}^{N-1} (N-i)|A^{i-1}B|_{S_2}^2 \right)(BB^*)^{-1}\right)\\
    &\leq |B|_{S_2}^2(BB^*)^{-1}\E_{\Pi_{s,\epsilon}}\left(\sum \limits_{i=0}^{N-2} (N-1-i)|A|_{S_\infty}^{2i} \right)\\
    &\leq |B|_{S_2}^2(BB^*)^{-1}\left(\sum \limits_{i=0}^{N-2} (N-1-i)(s+\epsilon)^{2i} \right).
\end{align*}
Putting together all the above estimates, we obtain
\begin{equation*}\begin{array}{lll}
       \mathcal{E}_2(\mathcal{C}_s) \geq \inf \limits_{\hat{A}} \tr\left(\E_{\Pi_{s,\epsilon}} \left(A \nabla_A \log({\Pi_{s,\epsilon}}(A))^*\right)\left( \E_{\Pi_{s,\epsilon}}(I_{(x)}) + I_{\Pi_{s,\epsilon}}\right)^{-1}\right. \\ \left. \qquad\qquad\qquad\qquad\E_{\Pi_{s,\epsilon}} \left(A \nabla_A \log({\Pi_{s,\epsilon}}(A))^*\right)^* \right)\\
        \qquad\quad\geq d^2 \tr\left( |B|_{S_2}^2(BB^*)^{-1}\left(\sum \limits_{i=0}^{N-2} (N-1-i)(s+\epsilon)^{2i} \right) + \frac{2(d+1)(d+2)}{\epsilon^2}I_d \right)^{-1} \\
       \qquad\quad= \frac{ d^2}{|B|_{S_2}^2} \tr\left( (BB^*)^{-1}\left(\left(\sum \limits_{i=0}^{N-2} (N-1-i)(s+\epsilon)^{2i} \right)I_d + \frac{2(d+2)^2}{\epsilon^2}\frac{ BB^*}{|B|_{S_2}^2} \right)\right)^{-1} \\
       \qquad\quad \geq \frac{ d^2}{|B|_{S_2}^2} \tr\left( (BB^*)^{-1}\left(\left(\sum \limits_{i=0}^{N-2} (N-1-i)(s+\epsilon)^{2i} \right)I_d + \frac{2(d+2)^2}{\epsilon^2}I_d \right)\right)^{-1} \\
       \qquad\quad \geq d^2\frac{\tr(BB^*) }{|B|_{S_2}^2} \left(\left(\sum \limits_{i=0}^{N-2} (N-1-i)(s+\epsilon)^{2i} \right) + \frac{2(d+2)^2}{\epsilon^2} \right)^{-1} \\
       \qquad\quad \geq d^2\left(\sum \limits_{i=0}^{N-2} (N-1-i)(s+\epsilon)^{2i}  + \frac{2(d+2)^2}{\epsilon^2}\right)^{-1}.
       \end{array}
      \end{equation*}
\end{proof}
In the next corollary we give lower bounds for the minimax mean square error risk which correspond to different ranges of the sample size $N$ when $s\in [0,1)\cup (1,+\infty)$ and for the case $s=1$.
\begin{corollary}\label{Explicit minimax}
\begin{enumerate}
\item[(a)] If $s \in [0, 1)$, for all $ \alpha \in (0, 1)$ and $N$ such that
\begin{equation}\label{case-1}
    N \geq \frac{16(d+2)^2}{\alpha(1-s)^2},
\end{equation}
we have
\begin{align}\label{case-1-1}
       \mathcal{E}_2(\mathcal{C}_s) &\geq \frac{d^2(1-(s+1)^2/4)}{(1+\alpha)N}.
\end{align}
\item[(b)] If $s > 1$, for all $ \alpha \in (0, 1)$ and $N$ such that
\begin{align}\label{case-2}
    N &\geq \log_2\left(\frac{d+2}{\alpha}\right) + 3,
\end{align}
we have
\begin{align}\label{case-2-2}
       \mathcal{E}_2(\mathcal{C}_s) &\geq \frac{d^2((s+1)^2-1)^2}{(1+\alpha)(s+1)^{2N}}.
\end{align}
\item[(c)] If $s = 1$, we have
\begin{align}\label{case-2-3}
       \mathcal{E}_2(\mathcal{C}_1) \geq \frac{\log^2(d+2)}{3N^2(1+ 2/d)^2} .
\end{align}
\end{enumerate}
\end{corollary}
This corollary generalizes the results of the previous section by providing rates that apply to all estimators as opposed to just the least square estimator. These rates do not suffer from the shortcomings of the results in \cite{pmlr-v75-simchowitz18a} and \cite{Jedra19}, discussed in the introduction. Indeed, they offer a clear description of the minimax rate over all matrices since any matrix $A$ belongs to $\mathcal{C}_{s_{\min}(A)}$, whereas \cite{pmlr-v75-simchowitz18a} only provides rates for scaled orthogonal matrices. They also make appear the dimension factor missing in \cite{Jedra19}. Note also that the overall minimax rate is given for all $\alpha \in (0, 1)$ by 
\begin{align}
       \mathcal{E}_2(\mathcal{C}_0) \geq \frac{3d^2}{4(1+\alpha)N}, \quad \text{valid for} \quad 
    N \geq \frac{16(d+2)^2}{\alpha}.
\end{align} 
These results go even beyond by answering the questions mentioned in the introduction and show the existence of an estimation rate that is independent of the covariance structure provided by the matrix $B$ in \eqref{Eq: state space model}. They also show that the difficulty of the estimation problem in the class $\mathcal{C}_{s}$ is ultimately related to the parameter $s$ which is the least singular value lower bound. This provides a clear characterization of the phase transition phenomena discussed in the introduction. Indeed, the estimation problem over the class $\mathcal{C}_{s}$ becomes easier if $s$ increases and we witness the appearance of the three regimes $s<1$, $s = 1$, and $s>1$ as described in the corollary.

\begin{remark}\label{Important remak}
As in the results of Section 3, the rates provided by this corollary hold for specific sets of values for $N$. These values get arbitrarily large if $s$ is arbitrarily close to $1$ which is just an artifact of the proof that has nothing to do with the lower bound provided by Theorem \ref{minimax}. To avoid this issue we can reduce the size of the first set of values of $s$ to $s \in \left[0, 1 - \left(\frac{(d+2)^2}{N}\right)^{1/4} \right]$. Plugging in these values, the condition in \eqref{case-1} and the rate \eqref{case-1-1} become
\begin{equation*}
    N \geq \frac{16^2}{\alpha^2} (d+2)^2 \,\,\,  \text{implies}\,\,\,  \mathcal{E}_2(\mathcal{C}_s) \geq \frac{d^2(1-(s+1)^2/4)}{(1+\alpha)N}.
\end{equation*}
This clearly removes the issue of arbitrarily large $N$ if $s$ is close to $1$, but the drawback is in getting a rate $N^{-2}$ on the larger set $s \in \left[1+\left(\frac{(d+2)^2}{N} \right)^{1/4}, 1\right]$, which is too optimistic for a minimax lower bound in this case.
\end{remark}
\begin{proof}[Proof of Corollary \ref{Explicit minimax}]
Starting from Theorem \ref{minimax} and using the upper bound of the term $$\left(\sum \limits_{i=0}^{N-2} (N-1-i)(s+\epsilon)^{2i}  + \frac{2(d+2)^2}{\epsilon^2}\right)^{-1}$$ displayed in Lemma \ref{Estimate for double geometric series} in Appendix C, below, we obtain
      \begin{equation}\label{Eq: Main lower bound on the minimax risk}
         \mathcal{E}_2(\mathcal{C}_s) \geq  \begin{cases}
        d^2\left( \frac{N}{1-(s+\epsilon)^2} + \frac{1}{(1-(s+\epsilon)^2)^2} + \frac{2(d+2)^2}{\epsilon^2} \right)^{-1} &\text{if} \quad 0<s +\epsilon < 1,\\ \\
        d^2\left( \frac{(s+\epsilon)^{2N}}{((s+\epsilon)^2-1)^2} + \frac{2(d+2)^2}{\epsilon^2} \right)^{-1} &\text{if} \quad s + \epsilon > 1.
        \end{cases}
\end{equation}
 (a) \,\, For $s \in [0, 1)$, we take $\epsilon = \frac{1-s}{2}$ to have
\begin{align*}
       \mathcal{E}_2(\mathcal{C}_s) &\geq d^2\left(\frac{N}{1-(s+1)^2/4} + \frac{1}{(1-(s+1)^2/4)^2} + \frac{8(d+2)^2}{(1-s)^2}\right)^{-1} .
\end{align*}
Therefore, for all $ \alpha \in (0, 1)$, if we take $N$ such that 
\begin{equation}\label{N-a}
    N \geq \frac{2}{\alpha (1-(s+1)^2/4)} \vee \frac{16(d+2)^2(1-(s+1)^2/4)}{\alpha (1-s)^2},
\end{equation}
we obtain
\begin{align*}
       \mathcal{E}_2(\mathcal{C}_s) &\geq \frac{d^2(1-(s+1)^2/4)}{(1+\alpha)N}.
\end{align*}
Using the elementary inequality $1 -(s+1)^2/4 \geq (1-s)^2/4$ we see that the condition \eqref{N-a} is satisfied for the values of $N$ such that
\begin{equation*}
    N \geq \frac{16(d+2)^2}{\alpha(1-s)^2}.
\end{equation*}

\medskip
(b) \,\, For the case $s > 1$, we take $\epsilon = 1$ in \eqref{Eq: Main lower bound on the minimax risk} to get the lower bound
\begin{align*}
       \mathcal{E}_2(\mathcal{C}_s) &\geq \frac{d^2}{\frac{(s + 1)^{2N}}{((s + 1)^2-1)^2} + 2(d+2)^2 }.
\end{align*}
Hence, for all $ \alpha \in (0, 1)$, if we take $N$ such that 
\begin{equation}\label{N-b}
    2N \log(s+1) \geq 2 \log((s+1)^2-1) + 2\log(2) + 2\log(d+2) + \log\left(\frac{1}{\alpha}\right),
\end{equation}
we obtain
\begin{align*}
       \mathcal{E}_2(\mathcal{C}_s) &\geq \frac{d^2((s+1)^2-1)^2}{(1+\alpha)(s+1)^{2N}}.
\end{align*}
The condition \eqref{N-b} on the values of $N$ is satisfied if we take $N$ such that 
\begin{equation*}
    N \geq  2 + \frac{\log(2)}{2\log(s+1)} + \frac{\log(d+2)}{\log(s+1)} + \frac{1}{2\log(s+1)}\log\left(\frac{1}{\alpha}\right)\geq \log_2\left(\frac{d+2}{\alpha}\right) + 3.
\end{equation*}
\medskip
(c)\,\, For the case $s = 1$ we still have $s + \epsilon > 1$ the result in Eq. \eqref{Eq: Main lower bound on the minimax risk} is still valid and yields
\begin{align*}
       \mathcal{E}_2(\mathcal{C}_1) \geq d^2\left(\frac{(1 + \epsilon )^{2N}}{\epsilon^2} + \frac{2(d+2)^2}{\epsilon^2}\right)^{-1}  \geq d^2\left(e^{2N\epsilon }+ 2(d+2)^2\right)^{-1} \epsilon^2.
\end{align*}
Hence, for $ \epsilon =\frac{\log d+2}{N}$, we obtain 
\begin{align*}
       \mathcal{E}_2(\mathcal{C}_1) &\geq \frac{\log^2(d+2)}{3N^2(1+ 2/d)^2} .
\end{align*}
\end{proof}

\begin{proof}[Proof of lemma \ref{Computing the gradient in the new coordinates}]

Throughout this proof we use the notation $\{G\}(H) $ for the action of a linear operator $G$ on an element $H$ whenever this action is not easily represented by usual matrix multiplication or scalar product, $G$ will always be represented as the gradient $\nabla_A$ of some matrix value function computed at point $A$ of the Lie group of invertible matrices and $H$ will be an element of the corresponding Lie algebra which is identified with $\mathcal{M}_{d\times d}(\mathbb{R})$ such that $|H|_{S_2}<\infty$, a condition that is always satisfied since we are in finite dimensions (see \cite{Lang1999} for the differential geometric notions). 

\noindent For ease of notation we set $C:=\epsilon I_d- (\tilde{A}\tilde{A}^*)^{1/2},\,\, \tilde{\Pi}_{s,\epsilon}(A)= \det\left(\epsilon I_d- (\tilde{A}\tilde{A}^*)^{1/2}\right)^2$, where $\tilde{A}:= A-sI_d$. We note that
\begin{equation*}
    \nabla_A \log \Pi_{s,\epsilon}(A) = \nabla_A \log \tilde{\Pi}_{s,\epsilon}(A) = \frac{\nabla_A \tilde{\Pi}_{s,\epsilon}(A)}{\tilde{\Pi}_{s,\epsilon}(A)}.
\end{equation*}
We also define the bracket of two matrices $M$ and $N$ as
\begin{equation*}
    [M,N] := MN^* + NM^*.
\end{equation*}
Using the exponential map, we compute the gradient as follows:  
\begin{equation*}\begin{array}{lll}
    \innerl{\nabla_A \tilde{\Pi}_{s,\epsilon}(A)}{H} = \frac{d}{dt} \det\left(\left(\epsilon I_d- ( (e^{tHA^{-1}}A-sI_d)(e^{tHA^{-1}}A-sI_d)^*)^{1/2}\right)^2\right) \Big|_{t=0}\\
    \qquad= \tilde{\Pi}_{s,\epsilon}(A) \tr \left( C^{-2} \frac{d}{dt}\left(\epsilon^2I_d- ( (e^{tHA^{-1}}A-sI_d)(e^{tHA^{-1}}A-sI_d)^*)^{1/2} \right)^2\Big|_{t=0}\right)\\
    \qquad = \tilde{\Pi}_{s,\epsilon}(A) \tr \left( C^{-2} \left\{\nabla_A   \left(\epsilon I_d- (\tilde{A}\tilde{A}^*)^{1/2} \right)^2\right\}(H)\right) \\
    \qquad= \tilde{\Pi}_{s,\epsilon}(A) \tr \Big( C^{-1} \left\{\nabla_A  \left( \epsilon I_d- (\tilde{A}\tilde{A}^*)^{1/2} \right)\right\}(H)  + C^{-2}\left\{\nabla_A  \left( \epsilon I_d- (\tilde{A}\tilde{A}^*)^{1/2} \right)\right\}(H) C \Big)\\
    \qquad= 2\tilde{\Pi}_{s,\epsilon}(A) \tr \Big( C^{-1} \left\{\nabla_A  \left( \epsilon I_d- (\tilde{A}\tilde{A}^*)^{1/2} \right)\right\}(H)\Big).
    \end{array}
\end{equation*}

Therefore, 
\begin{equation*}
    \innerl{\nabla_A \log \Pi_{s,\epsilon}(A)}{H} = -2 \tr \Big( C^{-1} \left\{\nabla_A  \left(\tilde{A}\tilde{A}^* \right)^{1/2}\right\} (H)\Big).
\end{equation*}
Noting that 
\begin{align*}
    (\tilde{A} + H)(\tilde{A} + H)^* = \tilde{A}\tilde{A}^* + [\tilde{A},H] + HH^*, 
\end{align*}
we recognize the gradient from the linear part of this last expression by its action on $H$ thought the bracket $[\tilde{A},H]$ and we write 
\begin{align*}
    \left\{\nabla_A \tilde{A}\tilde{A}^* \right\} (H) := [\tilde{A},H].
\end{align*}
Let $\mathcal{W}$ be a non-singular matrix commuting with $\tilde{A}\tilde{A}^*$ so that we have
\begin{align*}
    \tr \Big( \mathcal{W}^{-1} \tilde{A}H^* \Big) &= \frac{1}{2}\tr \Big( \mathcal{W}^{-1} [\tilde{A},H] \Big) 
    = \frac{1}{2}\tr \Big( \mathcal{W}^{-1} \left\{\nabla_A  \tilde{A}\tilde{A}^* \right\} (H) \Big)\\
    &= \frac{1}{2}\tr \Big( \mathcal{W}^{-1} \left\{\nabla_A  \left(\tilde{A}\tilde{A}^* \right)^{1/2} \left(\tilde{A}\tilde{A}^* \right)^{1/2}\right\} (H) \Big)\\
    &= \tr \Big( \mathcal{W}^{-1}   \left(\tilde{A}\tilde{A}^* \right)^{1/2} \left\{\nabla_A\left(\tilde{A}\tilde{A}^* \right)^{1/2}\right\} (H) \Big).
\end{align*}
Taking
\begin{equation*}
    \mathcal{W}:= \left(\tilde{A}\tilde{A}^* \right)^{1/2} C,
\end{equation*}
we have
\begin{align*}
    \tr \Big( C^{-1} \left\{\nabla_A\left(\tilde{A}\tilde{A}^* \right)^{1/2}\right\} (H) \Big) =
\tr \Big( C^{-1} \left(\tilde{A}\tilde{A}^*\right)^{-1/2}\tilde{A}H^* \Big). 
\end{align*}
Therefore,
\begin{equation*}
    \innerl{\nabla_A \log \Pi_{s,\epsilon}(A)}{H} = - 2\tr \Big( C^{-1}  \left(\tilde{A}\tilde{A}^*\right)^{-1/2}\tilde{A}H^*\Big).
\end{equation*}
Hence, we obtain 
\begin{align*}
    \nabla_A \log \Pi_{s,\epsilon}(A) &= - 2 \left( \epsilon I_d- (\tilde{A}\tilde{A}^*)^{1/2} \right)^{-1} (\tilde{A}\tilde{A}^*)^{-1/2} (A-sI_d)\\
    &= - 2 \left( \epsilon UU^*- (U\Sigma VV^*\Sigma U^*)^{1/2} \right)^{-1} (U\Sigma VV^*\Sigma U^*)^{-1/2} U\Sigma V\\
    &= - 2 U\left( \epsilon I_d- \Sigma \right)^{-1} U^*U\Sigma^{-1} U^* U\Sigma V= - 2 U\left( \epsilon I_d- \Sigma \right)^{-1} V^*.
    \end{align*}
\end{proof}

\section*{Appendix A: Proof of the main probabilistic inequalities}
The proof of both Proposition \ref{bound on the spectrum for the sample covariance of a state space model} and Proposition \ref{bound on the multiplication process of a state space models} relay on a generic chaining argument. We start by recalling few concepts from the generic chaining literature \cite{talagrand2006generic,gine_nickl_2015} to fix the notation. Let $(\mathcal{A},d)$ be a metric space. The distance of a point $t \in \mathcal{A}$ to a subset $\mathbb{A} \subseteq \mathcal{A}$ is defined as
\begin{equation*}
    d(t,\mathbb{A}) = \inf \limits_{s \in \mathbb{A} } d(t,s).
\end{equation*}
The diameter of the set $\mathbb{A}$ is 
\begin{equation*}
    \Delta(\mathbb{A}) = \sup \limits_{(s,t) \in \mathbb{A}^2} d(t,s),
\end{equation*}
and the covering number $N(\mathcal{A},d,u)$ is the smallest number of balls in $(\mathcal{A},d)$ of radius less than $u$  needed to cover $\mathcal{A}$ (\emph{i.e.}, whose union includes $\mathcal{A}$). A ball of center $c \in \mathcal{A}$ and radius $r\ge 0$ with respect to a distance $d$ or a metric $|\cdot|_d$ will be denoted $B_d(c,r)$ or $B_{|\cdot|_d}(c,r)$, respectively.

The gamma-$\alpha$ functional $\gamma_\alpha(\mathcal{A},d)$ for the metric space $(\mathcal{A}, d)$ and its corresponding upper bound by the Dudley chaining integral are defined as follows: 
\begin{align}
    \gamma_\alpha(\mathcal{A},d) := \inf \sup \limits_{t \in \mathcal{A}} \sum \limits_{r= 0}^{\infty} 2^{r/\alpha} d(t,\mathbb{A}_r) 
    \lesssim \int_{0}^{ \Delta(\mathcal{A})} (\log N(\mathcal{A}, d,u))^{1/\alpha}du,\label{ekv}
\end{align}
where the infimum is taken over all sequences of sets $(\mathbb{A}_r)_{r \in \mathbb{N}}$ in $\mathcal{A}$ with $|\mathbb{A}_0|=1$ and $|\mathbb{A}_r| \leq 2^{2^r}$ (\cite{talagrand2006generic}). If $d(x,y) = \lVert x-y\rVert$ for some norm $\lVert\cdot\rVert$ as it is usually the case, we also use the notation $\gamma_\alpha(\mathcal{A},\lVert\cdot\rVert)$ for $\gamma_\alpha(\mathcal{A},d)$.
\begin{proof}[Proof of Proposition \ref{bound on the spectrum for the sample covariance of a state space model} ]
The inequality will be obtained as a result of using the following generic chaining theorem due to ~\citet[Theorem ~$3.5$]{dirksen2015}
\begin{theorem}[Dirksen \cite{dirksen2015}]
Let $\mathbb{B}$ be a set of matrices and $\varepsilon = (\varepsilon_0,\dots,\varepsilon_{N-1})$ a column vector of independent zero mean standard normal vectors. Then for all $t \geq 1$ we have
\begin{equation}
    \mathbb{P}\left( \sup \limits_{\mathcal{B} \in \mathbb{B}} ||\mathcal{B} \varepsilon|_2^2-\E(|\mathcal{B} \varepsilon |_2^2)|\geq C \left( E +\sqrt{t} V + t U \right)\right) \leq \exp{(-t)},
\end{equation}
where, 
\begin{align*}
    E = \gamma_2^2(\mathbb{B}, |\cdot|_{S_\infty}) + \Delta_{S_2} (\mathbb{B}) \gamma_2(\mathbb{B}, |\cdot|_{S_\infty}), \quad
    V = \Delta_{S_4}^2 (\mathbb{B}), \quad
    U= \Delta_{S_\infty}^2  (\mathbb{B}).
\end{align*}
\end{theorem}
We want to bound
\begin{equation*}
\sup_{|u|_2 \leq1} | u^* \Psi^{-1/2} (\sum \limits_{i=1}^{N-1} x_ix_i^*) \Psi^{-1/2} u- u^* u |.
\end{equation*}
To this end, define the bloc lower triangular Toeplitz matrix with the lower bloc coefficients
\begin{equation*}
     (\mathcal{B}_u)_{i,j} = B^*  A^{*i \vee j - 1} \Psi^{-1/2} u.
\end{equation*}
Note that
\begin{equation*}
\begin{array}{lll}
    u^* \Psi^{-1/2} (\sum \limits_{i=1}^{N-1} x_ix_i^*) \Psi^{-1/2} u 
    = \sum \limits_{k = 1}^{N-1} (\sum \limits_{i= 0}^{k-1} u^* \Psi^{-1/2} A^{k-1-i} B \varepsilon_i)(\sum \limits_{j= t}^N u^* \Psi^{-1/2} A^{k-1-j} B\varepsilon_j)^* \\
    \qquad \qquad =  \sum \limits_{i= 0}^{N-1} \sum \limits_{j= 0}^{N-1} \varepsilon_i^*  \left(\sum \limits_{t\geq i,j}^{N-1}  (u^* \Psi^{-1/2} A^{k-1-i} B )(u^*  \Psi^{-1/2} A^{k-1-j} B)^* \right)\varepsilon_j
    = |\mathcal{B}_u \varepsilon |_2^2.
    \end{array}
\end{equation*}
This defines a second order chaos. Since $\mathcal{B}_u$ is Toeplitz
\begin{align*}
    |\mathcal{B}_u|_{S_\infty} &\leq   \sup_{s \in [0\ 1]} |B^* (\sum_{k=1}^{N-2}  A^{*k}  e^{j 2\pi k s} )\Psi^{-1/2} u|_2\\
    &\leq |u|_2 \sup_{s \in [0\ 1]} | \Psi^{-1/2} (\sum_{t=0}^{N-2}  A^k  e^{j 2\pi k s} )B|_{S_\infty}.
\end{align*}
On the other hand
\begin{align*}
    |\mathcal{B}_u|_{S_2} &= ( \sum \limits_{i= 1}^{N-1} \sum \limits_{k = 0}^{i-1}  (u^* \Psi^{-1/2} A^{i-1-k} B  )(u^* \Psi^{-1/2} A^{i-1-k} B)^* )^{1/2}\\
    &= ( u^* \Psi^{-1/2} \sum \limits_{i= 1}^{N-1} \sum \limits_{k = 0}^{i-1}  ( A^{i-1-k} B B^* A^{*i-1-k}  ) \Psi^{-1/2} u )^{1/2} = |u|_2.
\end{align*}
This also implies that
\begin{align*}
    |\mathcal{B}_u|_{S_4} &= (\innerl{\mathcal{B}_u^* \mathcal{B}_u}{\mathcal{B}_u^* \mathcal{B}_u})^{1/4} \leq (|\mathcal{B}_u^* \mathcal{B}_u|_{S_1}|\mathcal{B}_u^* \mathcal{B}_u|_{S_\infty})^{1/4} \\
    &= (|\mathcal{B}_u|_{S_2}| \mathcal{B}_u|_{S_\infty})^{1/2}=(|u|_2| \mathcal{B}_u|_{S_\infty})^{1/2}.
\end{align*}
Defining $\mathbb{B}_2 = \{u, \,\,\,  \ | u|_2 \leq 1\}$ and $\mathcal{L}_{A,B} = \sup_{s \in [0\ 1]} | \Psi^{-1/2} (\sum_{k=0}^{N-2} A^k  e^{j 2\pi k s} ) B|_{S_\infty}^2$ we obtain
\begin{equation*}
    \Delta_{S_\infty}(\mathbb{B}) \leq \mathcal{L}_{A,B} \ ,\Delta_{S_2}(\mathbb{B}) = 1 \quad \textrm{and} \quad \Delta_{S_4}(\mathbb{B}) \leq \mathcal{L}_{A,B}^{1/2}.
\end{equation*}
Consider the Gaussian vector $g \sim \mathcal{N}(0, \mathcal{L}_{A,B}^2 I_d)$ and the Gaussian process $G_u = \innerl{g}{u}$ indexed over $\mathbb{B}_2$. The process is zero mean with covariance structure 
\begin{equation*}
    \E \innerl{g}{u}\innerl{g}{v} = \E \innerl{v}{gg^* u} =  \innerl{v}{u} \mathcal{L}_{A,B}^2.
\end{equation*}
Using Talagrand Majorizing Measure Theorem, we have
\begin{align*}
    \gamma_2(\mathbb{B}_2,|\mathcal{B}_u|_{S_\infty} ) &\leq
    \gamma_2(\mathbb{B}_2,|u|_2 \mathcal{L}_{A,B} )= \E \sup_{u \in \mathbb{B}} \innerl{g}{u} = \E |g|_2\leq (\E |g|_2^2)^{1/2} = d^{1/2} \mathcal{L}_{A,B}. 
\end{align*}
We can apply the generic chaining result of ~\citet[Theorem ~$6.5$]{dirksen2015} to obtain the desired bound: With probability $1-e^{-t}$
\begin{align*}
|\Psi^{-1/2} (\sum \limits_{i=1}^{N-1} x_ix_i^*) \Psi^{-1/2} - I_d|_{S_\infty} &\lesssim d\mathcal{L}_{A,B} + (d\mathcal{L}_{A,B})^{1/2} + t\mathcal{L}_{A,B} + (t\mathcal{L}_{A,B})^{1/2}\\
&\lesssim (d\vee t)\mathcal{L}_{A,B} + (d\vee t)^{1/2}\mathcal{L}^{1/2}_{A,B}.
\end{align*}
\end{proof}

\begin{proof}[Proof of Proposition \ref{bound on the spectrum for the sample covariance of a state space model}]
We start by expressing the norm as a sumpremum
\begin{equation*}\begin{array}{lll}
    \E\left|\Psi^{-1/2}(\sum \limits_{i=1}^{N-1} x_i\varepsilon_i^*)\right|_{S_\infty}^2= \E \sup \limits_{\substack{|u|_2 = 1\\
    |v|_2 = 1}} \innerl{u}{\Psi^{-1/2}(\sum \limits_{i=1}^{N-1} x_i\varepsilon_i^*) v}^2 \\
    \qquad\qquad \qquad \quad  =  \E\sup \limits_{\substack{|u|_2 = 1\\
    |v|_2 = 1}} \left( \sum \limits_{i=1}^{N-1} \sum \limits_{j=0}^{i-1} \innerl{u}{\varepsilon_i}\innerl{v}{\Psi^{-1/2} A^{i-j-1} B\varepsilon_j}\right)^2.
    \end{array}
\end{equation*}
We note as well that
\begin{equation*}
    \E\left( \sum \limits_{i=1}^{N-1} \sum \limits_{j=0}^{i-1} \innerl{u}{\varepsilon_i}\innerl{v}{\Psi^{-1/2} A^{i-j-1} B\varepsilon_j} \right) = 0.
\end{equation*}
Hence, the problem reduces to a bound on the second moment of the supremum of a second order Gaussian chaos given by $\varepsilon^* W_{u,v} \varepsilon$ over the set $(u,v) \in \mathbb{S} =  \mathbb{S}_2^{d-1}\times \mathbb{S}_2^{d-1}$ where $W_{u,v}$ is the bloc lower triangular matrix with constant bloc diagonals given by
{\small \begin{equation*}
    W_{u,v}=
    \begin{bmatrix}
    0 & & & &\\
    uv^* \Psi^{-1/2} B & 0 & & &\\
    uv^* \Psi^{-1/2} AB & uv^* \Psi^{-1/2} B &0 & & &\\
    & & & & &\\
    & & & & &\\
    uv^* \Psi^{-1/2} A^{N-2} B & \, uv^* \Psi^{-1/2}A^{N-3} B & \, uv^* \Psi^{-1/2} A^{N-4} B &\ldots &uv^* \Psi^{-1/2}B &  0 
    \end{bmatrix}
\end{equation*}
}
and $\varepsilon$ is the vector obtained from putting all the $\varepsilon_i$ together in one column vector. Hanson-Wright inequality (see \citet{rudelson2013}) yields that with probability $1-e^{-t}$ the following inequality holds
\begin{equation*}
    |\varepsilon^* W_{u_1,v_1} \varepsilon - \varepsilon^* W_{u_2,v_2} \varepsilon| \leq t^{1/2}|W_{u_1,v_1} - W_{u_2,v_2}|_{S_2} + t|W_{u_1,v_1} - W_{u_2,v_2}|_{S_\infty},
\end{equation*}
which in turn  provides a mixed tail control of the increment of the process $\varepsilon^* W_{u,v} \varepsilon$. Let us define the following distances $d_2$ and $d_\infty$ on the set $\mathbb{S} =\mathbb{S}_2^{d-1}\times \mathbb{S}_2^{d-1}$ making it a compact metric space. For all $(u_1,v_1)$ and $(u_2,v_2)$ in $\mathbb{S}$
\begin{align*}
    &d_2((u_1,v_1),(u_2,v_2)) = |W_{u_1,v_1} - W_{u_2,v_2}|_{S_2}\\
    &= \left( \tr(\sum \limits_{k = 0}^{N-2} (N-1-k)(A^k B)^* \Psi^{-1/2}(u_1v_1^* - u_2v_2^*)^*(u_1v_1^* - u_2v_2^*)\Psi^{-1/2} A^k B)\right)^{1/2}\\
    &= \left( \innerl{\Psi^{-1/2}(u_1v_1^* - u_2v_2^*)^*(u_1v_1^* - u_2v_2^*)\Psi^{-1/2}}{\sum \limits_{k = 0}^{N-2} (N-1-k)A^k B (A^k B)^*  )}\right)^{1/2}\\
    &=  \left( \innerl{\Psi^{-1/2}(u_1v_1^* - u_2v_2^*)^*(u_1v_1^* - u_2v_2^*)\Psi^{-1/2}}{\Psi  )}\right)^{1/2}= |u_1v_1^* - u_2v_2^*|_{S_2}
\end{align*}
and 
\begin{align}
    &d_\infty ((u_1,v_1),(u_2,v_2)) = |W_{u_1,v_1} - W_{u_2,v_2}|_{S_\infty}\nonumber \\
    &= \sup \limits_{s \in [0\ 1]} |(u_1v_1^* - u_2v_2^*) \Psi^{-1/2} \sum \limits_{k = 0}^{N-2} e^{j2\pi s k} A^k B|_{S_\infty} \nonumber\\
    &\leq \sup \limits_{s \in [0\ 1]} |\Psi^{-1/2} \sum \limits_{k = 0}^{N-2} e^{j2\pi s k}A^k B|_{S_\infty} |(u_1v_1^* - u_2v_2^*)|_{S_\infty} \nonumber\\ 
    &\leq \sup \limits_{s \in [0\ 1]} |\Psi^{-1/2} \sum \limits_{k = 0}^{N-2} e^{j2\pi s k}A^k B|_{S_\infty}(|u_1-u_2|_2 + |v_1-v_2|_2),\label{d infinit upper bound}
\end{align} 
where in the last step we used Lemma \ref{control of nuclear norm by l2 norm} that appears in Appendix C. We also define the set $\mathbb{S}_W = \{W_{u,v}\ |\ (u,v) \in \mathbb{S}\}$. The generic chaining result proved independently by~\citet[Theorem ~$2.2.23$]{talagrand2006generic} and ~\citet[Theorem ~$3.5$]{dirksen2015} provides us with the following bound on the second moment of the supremum of this second order Gaussian chaos: \begin{multline}\label{Eq: result 1}
    \left( \E\left|\Psi^{-1/2}(\sum \limits_{i=1}^{N-1} x_i\varepsilon_i^*)\right|_{S_\infty}^2 \right)^{1/2} = \left(\E\sup \limits_{\substack{|u|_2 = 1\\
    |v|_2 = 1}} (\varepsilon^* W_{u,v} \varepsilon)^2 \right)^{1/2} \\
    \lesssim \gamma_2(\mathbb{S}_W,|\cdot|_{S_2}) + \gamma_1(\mathbb{S}_W,|\cdot|_{S_\infty}) + 2\sup \limits_{\substack{|u|_2 = 1\\
    |v|_2 = 1}} \left(\E (\varepsilon^* W_{u,v} \varepsilon)^2 \right)^{1/2}.
\end{multline}
We compute the last tree terms starting from the last one. We have
{\small \begin{align*}
    &\E\left( (\varepsilon^* W_{u,v} \varepsilon)^2 \right)  = \E\left( \left( \sum \limits_{i = 1}^{N-1} \sum \limits_{j = 0}^{i-1} \innerl{\varepsilon_i}{u} \innerl{\Psi^{-1/2} A^{i-j-1} B\varepsilon_j}{v}\right)^2 \right)\\
    & = \E \left( \sum \limits_{i = 1}^{N-1} \sum \limits_{j = 0}^{i-1} \innerl{\varepsilon_i}{u} \innerl{\Psi^{-1/2} A^{i-j-1} B\varepsilon_j}{v}\right) \left( \sum \limits_{k = 1}^{N-1} \sum \limits_{l = 0}^{k-1} \innerl{\varepsilon_k}{u} \innerl{\Psi^{-1/2} A^{i-j-1} B\varepsilon_l}{v}\right)\\
    & = \E\left( \left( \sum \limits_{i = 1}^{N-1}  \innerl{\varepsilon_i}{u}^2 \sum \limits_{j = 0}^{i-1} \innerl{\Psi^{-1/2}  A^{i-j-1} B\varepsilon_j}{v}^2\right) \right)\\
    &= |u|_2^2 \left(  \innerl{\Psi^{-1/2}  ( \sum \limits_{i = 1}^{N-1}   \sum \limits_{j = 0}^{i-1}  A^{i-j-1} B(A^{i-j-1} B)^* ) \Psi^{-1/2} }{v v^*}\right)= |u|_2^2|v|_2^2 = |uv^*|_{S_2}^2.
\end{align*}
}
Therefore, the last term is
\begin{equation}\label{Eq: result 2}
  \sup \limits_{\substack{|u|_2 = 1\\
    |v|_2 = 1}}  \E (\varepsilon^* W_{u,v} \varepsilon)^2 = 1.
\end{equation}
We define the Gaussian process $G_{u,v}  = \innerl{G}{uv^*}$ and note that 
\begin{equation*}
    \E\left((G_{u,v} - G_{u,v})^2\right)^{1/2} = |u_1v_1^* - u_2v_2^*|_{S_2} =d_2((u_1,v_1),(u_2,v_2)). 
\end{equation*}
Hence, using Talagrand majorizing measure Theorem ~\cite[Chapter $2$]{talagrand2006generic} and a classical bound on the operator norm of Gaussian random matrices ~\cite[Chapter $7$]{vershynin_2018} we get
\begin{equation}\label{Eq: result 3}
    \gamma_2(\mathbb{S}_W,|\cdot|_{S_2}) = \gamma_2(\mathbb{S},d_2) \simeq \E\sup \limits_{(u,v) \in \mathbb{S}} \innerl{G}{uv^*} = \E|G|_{S_\infty} \leq 2d^{1/2}. 
\end{equation}
To estimate the last term $\gamma_1(\mathbb{S}_W,|\cdot|_{S_\infty})$, recall that from \eqref{d infinit upper bound} for all $(u,v)\in \mathbb{S}$ we have 
\begin{equation*}
|W_{u,v}|_{S_\infty} \leq  \sup \limits_{s \in [0\ 1]} |\Psi^{-1/2} \sum \limits_{k = 0}^{N-2} e^{j2\pi s k} A^k B|_{S_\infty} (|u_1-u_2|_2 + |v_1-v_2|_2), 
\end{equation*}
we define the distances $d_{\infty_1}$ and $d_{\infty_2}$ on $\mathbb{S}$ by 
\begin{equation*}
    d_{\infty_1}((u_1,v_1),(u_2,v_2)) = |u_1-u_2|_2 \quad \text{and} \quad d_{\infty_2}((u_1,v_1),(u_2,v_2)) = |v_1-v_2|_2.
\end{equation*}
Thus, we have the following norm domination $d_\infty \leq \mathcal{L}_{A,B}^{1/2} (d_{\infty_1} + d_{\infty_2})$ valid on $\mathbb{S}$. Using the sub-additive property and the scaling property of the $\gamma_1$ functional ~\cite[Chapter $2$]{talagrand2006generic} we can upper bound the $\gamma_1$ term by 
\begin{align*}
    &\gamma_1(\mathbb{S}_W,|\cdot|_{S_\infty}) \lesssim \gamma_1(\mathbb{S}_W,\mathcal{L}_{A,B}^{1/2} (d_{\infty_1} + d_{\infty_2}))\\
    &\lesssim \gamma_1(\mathbb{S}_W,\mathcal{L}_{A,B}^{1/2} d_{\infty_1}) + \gamma_1(\mathbb{S}_W,\mathcal{L}_{A,B}^{1/2} d_{\infty_2}) \lesssim \mathcal{L}_{A,B}^{1/2} \gamma_1(\mathbb{S}_2^{d-1}, d_{\infty_1}),
\end{align*}
A classical covering number estimate for the unit sphere gives $\mathcal{N}(\mathbb{S}_2^{d-1},d_{\infty_1} ,u) \leq \left(\frac{3}{u}\right)^{d}$ (see e.g. \cite{vershynin_2018}). Plugged in the last expression we obtain 
\begin{align*}
    \gamma_1(\mathbb{S}_W,|\cdot|_{S_\infty}) &\lesssim \mathcal{L}_{A,B}^{1/2} \gamma_1(\mathbb{S}, d_{\infty_1}) \lesssim  \mathcal{L}_{A,B}^{1/2} \left(\int_{0}^{2} \log(\mathcal{N}(\mathbb{S}_2^{d-1},d_{\infty_1} ,u))du \right)\\
    & \lesssim  \mathcal{L}_{A,B}^{1/2} d \int_{0}^{2} \log \frac{3}{u} du \lesssim \mathcal{L}_{A,B}^{1/2} d.
\end{align*}
Collectively, this last estimate with the results \eqref{Eq: result 1}, \eqref{Eq: result 1}, and \eqref{Eq: result 3} gives the desired bound
\begin{align*}
    \left(\E \left|\Psi^{-1/2}(\sum \limits_{i=1}^{N-1} x_i\varepsilon_i^*)\right|_{S_\infty}^2 \right)^{1/2} &\lesssim  \mathcal{L}_{A,B}^{1/2} d .
\end{align*}

\end{proof}

\section*{Appendix B: Proof of the change of variable formula}
\begin{proof}[Proof of Proposition \eqref{Change of measure theorem}]
The parametrizations of the set of matrices are related through the singular values decomposition given as $A - sI_d = U\Sigma V$ where
\begin{equation*}
    \Sigma = \diag(\sigma_1,\dots,\sigma_d).
\end{equation*}
We recall that the singular values $\sigma_1,\dots,\sigma_d$ are the positive square roots of the solutions of the polynomial equation $0 = \det(\lambda I_d - (A-sI)(A-sI)^*)$. Thus the set of matrices with non-distinct singular values can be identified with the set of such polynomials with non-distinct zeros. Since the prior $\Pi_{s,\epsilon}$ is absolutely continuous to the Lebesgue measure $dA$, this set is of measure 0. Hence in what follows we assume without loss of generality that all the singular values of $A$ are distinct.

Let the set of ordered singular values be
\begin{equation*}
    \mathbb{R}_>^d = \{((\sigma_1,\sigma_2,\dots,\sigma_d) \ | \ 0 < \sigma_1<\sigma_2<\dots<\sigma_d\},
\end{equation*}
which is a $d$-dimensional manifold on which we take the natural parametrization $\mathbb{R}_>^d \to \mathcal{M}_{d \times d}(\mathbb{R})$ for $\Sigma_{\sigma} = \diag(\sigma_1,\sigma_2,\dots,\sigma_d)$. Let also $O(d)$ the set of orthogonal matrices on which we impose that the first non-zero entry of each column is positive, this is a $M = d(d-1)/2$-dimensional manifold on which we take two parametrizations $p_u: \mathbb{R}^M \to O(d)$ for $U_{p_u}$ and $q_v: \mathbb{R}^M \to O(d)$ for $V_{q_v}$ (see \cite{Lang1999} for the differential geometric notions). We denote the parameter by
\begin{equation*}
    \theta = (p_u,s_1,\dots,s_d,p_v) \in  O(d) \times \mathbb{R}_>^d \times O(d).
\end{equation*}

In \eqref{rewriting of the prior}, we wrote the density $\Pi_{s,\epsilon}(A)$ as product of terms $ Z_{i,s,\epsilon}(\epsilon  - \sigma_i)^2 \mathds{1}(\sigma_i \in [0 \ \epsilon])$, for any integrable function of $A$, we have
\begin{align*}
    &\E(f(A)) = \int f(A) \Pi_{s,\epsilon}(A)dA = \int f(A_\theta) \Pi_{s,\epsilon}(A_\theta)\left|\frac{\partial A}{\partial \theta}(A_\theta)\right|d\theta \\
    &= \int f(A_\theta)  \left|\frac{\partial A}{\partial \theta}(A_\theta)\right|\left( \prod \limits_{i=1}^d Z_{i,s,\epsilon}(\epsilon - \sigma_i)^2\mathds{1}(\sigma_i \in [0 \ \epsilon])d\sigma_i\right)dp_udp_v,
\end{align*}
where $\left|\frac{\partial A}{\partial \theta}(A_\theta)\right|$ is the Jacobian determinant for the change of variable of $\theta$ to $A$.

 Thus, the Jacobian determinant $\left|\frac{\partial A}{\partial \theta}(A_\theta)\right|$ is the Jacobian of the map
\begin{equation*}
    \eta: (p_u,\sigma,q_v) \in O(d) \times \mathbb{R}_>^d \times O(d) \to \eta(p_u,\sigma,q_v) = U_{p_u}\Sigma_\sigma V_{q_v}^* \in \mathcal{M}_{d \times d}(\mathbb{R}).  
\end{equation*}
Since we consider $O(d)$ restricted to matrices with the first non-zero entry positive, this map is a smooth bijection. Differentiating the matrix $A$ with respect to the parameters we obtain
\begin{align*}
    \frac{\partial A}{\partial p_u} = \frac{\partial U}{\partial p_u} \Sigma V^*,\qquad \frac{\partial A}{\partial q_u} =  U \Sigma \frac{\partial V^*}{\partial q_v}, \qquad \frac{\partial A}{\partial \sigma} = U\frac{\partial \Sigma}{\partial \sigma}  V^*.
\end{align*}
Now, since $V \in O(d)$, we have
\begin{align*}
    0 = \frac{\partial VV^*}{\partial q_v} = \frac{\partial V}{\partial q_v} V^* + V\frac{V^*}{\partial q_v}. 
\end{align*}
Therefore,
\begin{align*}
     \frac{V^*}{\partial q_v} = -V^*\frac{\partial V}{\partial q_v} V^*,
     \end{align*}
Thus, we obtain the system
\begin{empheq}[left=\empheqlbrace]{equation}\label{Eq: Differential system for change of coordinates}
    \begin{aligned}
     & U^* \frac{\partial A}{\partial p_u} V= U^*\frac{\partial U}{\partial p_u} \Sigma, \\
     & U^*\frac{\partial A}{\partial q_u} V=   -\Sigma V^*\frac{\partial V}{\partial q_v},\\
     &U^* \frac{\partial A}{\partial \sigma} V= \frac{\partial \Sigma}{\partial \sigma}.
    \end{aligned}
\end{empheq}
From here onward we denote the vectorized version of any matrix $B$ by the symbol $\overrightarrow{B}$. We define the matrix representation of the linear operator $X \in \mathcal{M}_{d \times d}(\mathbb{R})$ representing the linear map $A \to B$ such that $B = U^* A V$ which in vector format is equivalent to $X \overrightarrow{A} = \overrightarrow{B}$. we have
\begin{align*}
    \innerl{X \overrightarrow{A}}{X \overrightarrow{A}} = \innerl{U^* A V}{U^* A V} = \innerl{A}{A} = \innerl{\overrightarrow{A}}{\overrightarrow{A}}.
\end{align*}
Hence, for all $A$ we have $\innerl{ \overrightarrow{A}}{X^* X \overrightarrow{A}} = \innerl{\overrightarrow{A}}{\overrightarrow{A}} $ and $|\det(X)|^2 = |\det(X^* X)| =1$. Moreover the system of equations \eqref{Eq: Differential system for change of coordinates} in matrix format becomes
\begin{multline}\label{Eq: Matrix format for change of coordinates}
    X\left[\frac{\partial \overrightarrow{A}}{\partial \sigma_1},\dots , \frac{\partial \overrightarrow{A}}{\partial \sigma_d} ,\frac{\partial \overrightarrow{A}}{\partial p_{u_1}}, \dots ,\frac{\partial \overrightarrow{A}}{\partial p_{u_M} } ,\frac{\partial \overrightarrow{A}}{\partial q_{v_1}}, \dots ,\frac{\partial \overrightarrow{A}}{\partial q_{v_M} } \right] =\\
    \left[\frac{\partial \overrightarrow{\Sigma}}{\partial \sigma_1},\dots , \frac{\partial \overrightarrow{\Sigma}}{\partial \sigma_d} ,\overrightarrow{S_{u_1}\Sigma}, \dots ,\overrightarrow{S_{u_M}\Sigma}  ,\overrightarrow{\Sigma S_{v_1}}, \dots ,\overrightarrow{\Sigma S_{v_M}}   \right]
\end{multline}
with the definitions $S_{u_i} = U^*\frac{\partial U}{\partial p_{u_i}}$ and $S_{v_i} = V^*\frac{\partial V}{\partial p_{v_i}}$. Since the Jacobin determinant of $X$ is $1$, by taking the determinant of both sides of equation \eqref{Eq: Matrix format for change of coordinates} we obtain $\left|\frac{\partial A}{\partial \theta}(A_\theta)\right| $ is equal to the determinant of the last matrix. The entries of the last matrix are 
\begin{align*}
    &\left(\frac{\partial \overrightarrow{\Sigma}}{\partial \sigma_k}\right)_{i,j} = \left(\frac{\partial \Sigma}{\partial \sigma_k}\right)_{i,j} = \delta_{i,j}\delta_{j,k}, \\
    &\left(\overrightarrow{S_{u_k}\Sigma} \right)_{i,j} =\left(U^*\frac{\partial U}{\partial p_{u_k}} \Sigma \right)_{i,j} = \sigma_j \left(U^*\frac{\partial U}{\partial p_{u_k}} \right)_{i,j},\\
    &\left(\overrightarrow{\Sigma S_{v_k}}\right)_{i,j} =\left(-\Sigma V^*\frac{\partial V}{\partial q_{v_k}} \right)_{i,j} = -\sigma_i\left( V^*\frac{\partial V}{\partial q_{v_k}} \right)_{i,j}.
\end{align*}
In matrix form, this defines a block matrix of the form
\begin{equation*}
    \begin{bmatrix}
    I_d & & \\
    & S_I & \\
    & & &S_{II}
    \end{bmatrix}
\end{equation*}
where $S_I \in \mathcal{M}_{M \times M}(\mathbb{R})$ is a matrix of columns given by \begin{equation*}
    \sigma_j\Big[ \left(U^*\frac{\partial U}{\partial p_{u_1}} \right)_{i,j}, \left(U^*\frac{\partial U}{\partial p_{u_2}} \right)_{i,j} , \dots ,  \left(U^*\frac{\partial U}{\partial p_{u_M}} \right)_{i,j}\Big]
\end{equation*} 
which depends only on the $\sigma_i$ and $p_u$ and $S_{II} \in \mathcal{M}_{M \times M}(\mathbb{R})$ is a matrix of columns given by 
\begin{equation*}
    \sigma_i\Big[ -\left( V^*\frac{\partial V}{\partial q_{v_1}} \right)_{i,j}, -\left( V^*\frac{\partial V}{\partial q_{v_2}} \right)_{i,j} , \dots ,-\left( V^*\frac{\partial V}{\partial q_{v_M}} \right)_{i,j}\Big].
\end{equation*}
which depends only on the $\sigma_i$ and $p_v$. From the above considerations, it is easy to see that 
\begin{align*}
    \left|\frac{\partial A}{\partial \theta}(A_\theta)\right| = \prod \limits_{i= 1}^d \sigma_i^{d-1} g(p_u)h(p_v),
\end{align*}
where $g(p_u)$ is the absolute value of the determinant of the matrix with columns

$$\Big[ \left(U^*\frac{\partial U}{\partial p_{u_1}} \right)_{i,j}, \left(U^*\frac{\partial U}{\partial p_{u_2}} \right)_{i,j} , \dots ,  \left(U^*\frac{\partial U}{\partial p_{u_M}} \right)_{i,j}\Big]$$
and
$h(p_v)$ is the absolute value of the determinant of the matrix with columns
$$\Big[ -\left( V^*\frac{\partial V}{\partial q_{v_1}} \right)_{i,j}, -\left( V^*\frac{\partial V}{\partial q_{v_2}} \right)_{i,j} , \dots ,-\left( V^*\frac{\partial V}{\partial q_{v_M}} \right)_{i,j}\Big].
$$
Therefore, with $dU = g(p_u)dp_u$ and $dV = h(p_v)dp_v$, we have 
\begin{align*}
    &\E(f(A)) = \int f(A) \Pi_{s,\epsilon}(A)dA \\
    &= \int f(A_\theta) \Pi_{s,\epsilon}(A_\theta)\prod \limits_{i= 1}^d \sigma_i^{d-1} g(p_u)h(p_v)d\theta \\
    &= \int f(A_\theta)  \left( \prod \limits_{i=1}^d Z_{i,s,\epsilon}(\epsilon - \sigma_i)^2\sigma_i^{d-1}\mathds{1}(\sigma_i \in [0 \ \epsilon])d\sigma_i\right)g(p_u)h(p_v)dp_udp_v\\
    &= \int f(A_\theta)  \left( \prod \limits_{i=1}^d Z_{i,s,\epsilon}(\epsilon - \sigma_i)^2\sigma_i^{d-1}\mathds{1}(\sigma_i \in [0 \ \epsilon])d\sigma_i\right)dUdV.
\end{align*}
Moreover, recalling that
\begin{equation*}
    \Pi_{s,\epsilon}(A) = Z_{s,\epsilon} \det(\epsilon I_d - ((A-sI_d)(A-sI_d)^*)^{1/2})^2 \mathds{1}(A \in B_\infty\left(s, \epsilon \right))
\end{equation*}
is the density of a probability measure supported on $B_\infty\left(s, \epsilon \right)$ and taking $f(A) = \mathds{1}\{A \in B_\infty\left(s, \epsilon \right)\}$ we get
\begin{equation*}
    1 = \int_{B_\infty\left(s, \epsilon \right)}   \left( \prod \limits_{i=1}^d Z_{i,s,\epsilon}(\epsilon - \sigma_i)^2\sigma_i^{d-1}\mathds{1}(\sigma_i \in [0 \ \epsilon])d\sigma_i\right)dUdV,
\end{equation*}
meaning that $dU$ and $dV$ are both probability measures and for all $i \in [0 \ d]$ we have
\begin{equation*}
    Z_{i,s,\epsilon} = \left(\int \limits_0^\epsilon (\epsilon - \sigma_i)^2\sigma_i^{d-1}d\sigma_i\right)^{-1} =  \frac{d(d+1)(d+2)}{2\epsilon^{d+2}}.
\end{equation*}

\end{proof}

\section*{Appendix C: Technical results}
\begin{proposition}\label{control of nuclear norm by l2 norm}
For $(u_1,v_1)$ and $(u_2,v_2)$ in $\mathbb{S}_2^{d-1}\times \mathbb{S}_2^{d-1}$ the following norm inequality holds:
\begin{equation*}
    |u_1v_1^* - u_2v_2^*|_{S_\infty} \leq |u_1 - u_2|_2 + |v_1 -v_2|_2. 
\end{equation*}
\end{proposition}
\begin{proof}
Take $(u_1,v_1)$ and $(u_2,v_2)$ both in $\mathbb{S}_2^{d-1}\times \mathbb{S}_2^{d-1}$ and note that
    \begin{align*}
        |u_1v_1^* - u_2v_2^*|_{S_\infty}^2 &= \sup \limits_{a \in \mathbb{S}_2^{d-1}} | u_1\innerl{v_1}{a} - u_2\innerl{v_2}{a}|_2^2\\
        &= \sup \limits_{a \in \mathbb{S}_2^{d-1}} | u_1|_2^2\innerl{v_1}{a}^2 + | u_2|_2^2\innerl{v_2}{a}^2 - 2\innerl{u_1}{u_2}\innerl{v_1}{a}\innerl{v_2}{a}\\
        &= \sup \limits_{a \in \mathbb{S}_2^{d-1}} \innerl{v_1}{a}^2 + \innerl{v_2}{a}^2 +(|u_1-u_2|_2^2-2)\innerl{v_1}{a}\innerl{v_2}{a}\\
        &= \sup \limits_{a \in \mathbb{S}_2^{d-1}} \innerl{v_1 - v_2}{a}^2 +|u_1-u_2|_2^2\innerl{v_1}{a}\innerl{v_2}{a}\\
        &\leq \sup \limits_{a \in \mathbb{S}_2^{d-1}} \innerl{v_1 - v_2}{a}^2 +|u_1-u_2|_2^2\\
        &\leq  |v_1-v_2|_2^2 +|u_1-u_2|_2^2.
    \end{align*}
Taking the square root we obtain
    \begin{equation*}
        |u_1v_1^* - u_2v_2^*|_{S_\infty} \leq ( |u_1-u_2|_2^2 + |v_1-v_2|_2^2)^{1/2} \leq  |u_1-u_2|_2 + |v_1-v_2|_2.
    \end{equation*}
\end{proof}
The next lemma provides lower and upper estimates for $\mathcal{L}_{A,B}$ for a general linear state space model parametrized by a matrix $A$.

\begin{lemma}\label{Upper and lower bounds for LA}
If the matrix $A$ satisfies Assumption \ref{assumption} and the matrix $B$ is full rank, then for all $\alpha \in (0 \ 1)$ and $N$ such that
\begin{equation*}
    N \geq \frac{1}{\alpha(1-s_{\min}^2(A_u^{-1})\vee s_{\min}^2(A_s) )},
    \end{equation*}
we have 
\begin{equation*}
    \mathcal{L}_{A,B} \leq \frac{\left(\frac{1}{1-|A_u^{-1}|_{S_\infty}} + \frac{1}{1-|A_s|_{S_\infty}} + 1 \right)
    }{\frac{(1-\alpha)N}{1-s_{\min}^2(A_u^{-1})} \wedge \frac{(1-\alpha)N}{1-s_{\min}^2(A_s)} \wedge \frac{1}{3}} \cond^2(\tilde{B}).
\end{equation*}

\end{lemma}

\begin{proof}
Define
\begin{equation*}
    D = \begin{bmatrix}
    A_u^{-N}& & \\
    &I_s & \\
    & &\frac{1}{N}I_{\text{ls}}
    \end{bmatrix}S^{-1} \quad \text{and} \quad \Psi_D = D\Psi D^*.
\end{equation*}
Observe that
\begin{align}
    \mathcal{L}_{A,B} &= \sup_{s \in [0,\ 1]} |B^* (\sum_{k=0}^N  A^{*k}  e^{-j 2\pi k s} )D^* (D^{*-1}\Psi^{-1} D^{-1}) D(\sum_{k=0}^N  A^k  e^{j 2\pi k s} ) B|_{S_\infty} \nonumber \\
    &= \sup_{s \in [0,\ 1]} | \Psi_D^{-1/2} (\sum_{k=0}^N  D\tilde{A}^k  e^{j 2\pi k s} ) \tilde{B}|_{S_\infty}^2 \leq |\Psi_D^{-1} |_{S_\infty}\sup_{s \in [0\ 1]} | (\sum_{k=0}^{N-2}  D\tilde{A}^k  e^{j 2\pi k s} ) \tilde{B}|_{S_\infty}^2 \nonumber\\
    &= \frac{ \sup_{s \in [0\ 1]} | \sum_{k=0}^{N-2}  D \tilde{A}^k  e^{j 2\pi k s} |_{S_\infty}^2 |\tilde{B}|_{S_\infty}^2}{\lambda_{\min}(\Psi_D )}. \label{Eq: upper bound on the lipschitz constant for the evolution}
\end{align}
An upper bound for the supremum appearing in the last expression is obtained through a separation of the contribution of the each of the stable, unstable, and limit stable parts of the matrix $A$ as follows: 
\begin{align*}
    \sup_{s \in [0\ 1]} | \sum_{k=0}^{N-2}  D\tilde{A}^k  e^{j 2\pi k s} |_{S_\infty} &= \sup_{s \in [0\ 1]} \left| \sum_{k=0}^{N-2} e^{j 2\pi k s} \begin{bmatrix}
    {(A_u^{-1})}^{N-2-k}& & \\
    &A_s^k \quad & \\
    & &\quad \frac{1}{N-1}A_{\text{ls}}^k 
    \end{bmatrix}  \right|_{S_\infty} \\
    &= \sup_{s \in [0\ 1]} \left|  \begin{bmatrix}
    \sum_{k=0}^{N-2} e^{j 2\pi k s} {(A_u^{-1})}^{N-2-k}\\
    \sum_{k=0}^{N-2} e^{j 2\pi k s} A_s^k \\
    \sum_{k=0}^{N-2} e^{j 2\pi k s} \frac{1}{N-1}A_{\text{ls}}^k
    \end{bmatrix}\right|_{S_\infty} \\
    &\leq \left(\frac{1}{1-|A_u^{-1}|_{S_\infty}} + \frac{1}{1-|A_s|_{S_\infty}} + 1 \right)|B|_{S_\infty}.
\end{align*}
Consider the matrices
$$
\Theta(i,k):=\begin{bmatrix}
    (A_u^{-1})^{N-2-i+k}& & \\
    &A_s^{i-1-k} & \\
    & & \frac{1}{N-1}A_{\text{ls}}^{i-1-k} 
    \end{bmatrix}.
$$
The term involving the least eigenvalue in \eqref{Eq: upper bound on the lipschitz constant for the evolution} can be simplified by noting that
\begin{align*}
    &\lambda_{\min}\left( \sum_{i=1}^{N-1}\sum_{k=0}^{i-1} \Theta(i,k) \tilde{B} \tilde{B}^*  \Theta(i,k)^*\right) \geq \lambda_{\min} (\tilde{B} \tilde{B}^*) \sum_{i=1}^{N-1}\sum_{k=0}^{i-1} \lambda_{\min} (\Theta(i,k) \Theta(i,k)^*) \\
    &\geq s_{\min}^2(\tilde{B})\left(\sum \limits_{i=0}^{N-2} (N-1-i)s_{\min}^{2i}(A_u^{-1}) \wedge \sum \limits_{i=0}^{N-2} (N-1-i)s_{\min}^{2i}(A_s) \wedge \frac{1}{3} \right)\\
    &\geq s_{\min}^2(\tilde{B}) \left(\frac{(1-\alpha)N}{1-s_{\min}^2(A_u^{-1})} \wedge \frac{(1-\alpha)N}{1-s_{\min}^2(A_s)} \wedge \frac{1}{3} \right),
\end{align*}
where in the last inequality we used the lower estimate from Lemma \ref{Estimate for double geometric series} in this Appendix valid for
\begin{equation*}
    N \geq \frac{1}{\alpha(1-s_{\min}^2(A_u^{-1}))} \vee \frac{1}{\alpha(1-s_{\min}^2(A_s))}.
\end{equation*}
This last estimate together with Eq. \eqref{Eq: upper bound on the lipschitz constant for the evolution}  give us the result. Namely,
\begin{equation*}
    \mathcal{L}_{A,B} \leq \frac{\left(\frac{1}{1-|A_u^{-1}|_{S_\infty}} + \frac{1}{1-|A_s|_{S_\infty}} + 1 \right)
    }{\frac{(1-\alpha)N}{1-s_{\min}^2(A_u^{-1})} \wedge \frac{(1-\alpha)N}{1-s_{\min}^2(A_s)} \wedge \frac{1}{3}} \frac{|\tilde{B}|^2_{S_\infty}}{s_{\min}^2(\tilde{B})}.
\end{equation*}
\end{proof}
\begin{lemma}\label{Estimate for double geometric series} We have 
    \begin{itemize}
    \item[(a)] \begin{equation*}
        \sum \limits_{i=0}^{N-2} (N-1-i)a^i \leq \begin{cases}
        \frac{N}{1-a} + \frac{1}{(1-a)^2}\quad &\text{if} \quad a \in [0, 1),\\ 
        \frac{a^N}{(a-1)^2}\quad &\text{if} \quad a > 1.
        \end{cases}
    \end{equation*}
    
   \item[(b)] For $a$ and $\alpha \in (0\ 1)$, 
    \begin{equation*}
        \sum \limits_{i=0}^{N-2} (N-1-i)a^i \geq \frac{(1-\alpha) N}{1-a} \quad 
        \text{for} \quad N \geq  \frac{1}{\alpha(1-a)}.
    \end{equation*}
    \end{itemize}
\end{lemma}
\begin{proof}
(a)\,\, Let $b = 1/a$, since $a \neq 1$ and $a >0$, we have
    \begin{align*}
        \sum \limits_{i=0}^{N-2} (N-1-i)a^i &= a^{N-2} \sum \limits_{i=1}^{N-1} ib^{i-1} = a^{N-2} \sum \limits_{i=1}^{N-1} \frac{db^i}{db} = a^{N-2} \sum \limits_{i=1}^{N-1} \frac{dx^i}{dx} \big|_{x = b}\\
        & = a^{N-2}  \frac{d}{dx}\sum \limits_{i=0}^{N-1} x^i \big|_{x = b} = a^{N-2}  \frac{d}{dx}\frac{x^N - 1}{x-1} \big|_{x = b}\\
        & = a^{N-2}  \frac{Nb^{N-1}(b-1) - (b^N -1)}{(b-1)^2} \\
        & \leq \frac{Nb(b-1) + b^{-(N-2)})}{(b-1)^2}\\
        &\leq \begin{cases}
        \frac{N}{1-a} + \frac{a^N}{(1-a)^2}\quad &\text{if} \quad a \in (0, 1),\\
        \frac{a^N}{(1-a)^2}\quad &\text{if} \quad a > 1.
        \end{cases}
    \end{align*}
   (b)\,\,  We note that the last equality also gives, for all $\alpha \in (0,1)$, 
    \begin{align*}
        \sum \limits_{i=0}^{N-2} (N-1-i)a^i & =   \frac{a^N + (N-1) - Na}{(1-a)^2} \geq  \frac{(1-\alpha) N}{1-a} \quad 
        \text{for} \quad N \geq  \frac{1}{\alpha(1-a)}.
    \end{align*}
\end{proof}

\bibliographystyle{elsarticle-harv} 
\bibliography{main}

\begin{thebibliography}{19}
\expandafter\ifx\csname natexlab\endcsname\relax\def\natexlab#1{#1}\fi
\providecommand{\url}[1]{\texttt{#1}}
\providecommand{\href}[2]{#2}
\providecommand{\path}[1]{#1}
\providecommand{\DOIprefix}{doi:}
\providecommand{\ArXivprefix}{arXiv:}
\providecommand{\URLprefix}{URL: }
\providecommand{\Pubmedprefix}{pmid:}
\providecommand{\doi}[1]{\href{http://dx.doi.org/#1}{\path{#1}}}
\providecommand{\Pubmed}[1]{\href{pmid:#1}{\path{#1}}}
\providecommand{\bibinfo}[2]{#2}
\ifx\xfnm\relax \def\xfnm[#1]{\unskip,\space#1}\fi
\bibitem[{Adamczak et~al.(2011)Adamczak, Litvak, Pajor and
  Tomczak-Jaegermann}]{ADAMCZAK2011195}
\bibinfo{author}{Adamczak, R.}, \bibinfo{author}{Litvak, A.E.},
  \bibinfo{author}{Pajor, A.}, \bibinfo{author}{Tomczak-Jaegermann, N.},
  \bibinfo{year}{2011}.
\newblock \bibinfo{title}{Sharp bounds on the rate of convergence of the
  empirical covariance matrix}.
\newblock \bibinfo{journal}{Comptes Rendus Mathematique} \bibinfo{volume}{349},
  \bibinfo{pages}{195--200}.
\bibitem[{Chan and Wei(1988)}]{10.1214/aos/1176350711}
\bibinfo{author}{Chan, N.H.}, \bibinfo{author}{Wei, C.Z.},
  \bibinfo{year}{1988}.
\newblock \bibinfo{title}{{Limiting Distributions of Least Squares Estimates of
  Unstable Autoregressive Processes}}.
\newblock \bibinfo{journal}{The Annals of Statistics} \bibinfo{volume}{16},
  \bibinfo{pages}{367 -- 401}.
\bibitem[{Dirksen(2015)}]{dirksen2015}
\bibinfo{author}{Dirksen, S.}, \bibinfo{year}{2015}.
\newblock \bibinfo{title}{Tail bounds via generic chaining}.
\newblock \bibinfo{journal}{Electron. J. Probab.} \bibinfo{volume}{20},
  \bibinfo{pages}{1--29}.
\bibitem[{Djehiche et~al.(2021)Djehiche, Mazhar and Rojas}]{10.3150/20-BEJ1262}
\bibinfo{author}{Djehiche, B.}, \bibinfo{author}{Mazhar, O.},
  \bibinfo{author}{Rojas, C.R.}, \bibinfo{year}{2021}.
\newblock \bibinfo{title}{{Finite impulse response models: A non-asymptotic
  analysis of the least squares estimator}}.
\newblock \bibinfo{journal}{Bernoulli} \bibinfo{volume}{27},
  \bibinfo{pages}{976 -- 1000}.
\bibitem[{Gassiat et~al.(2013)Gassiat, Pollard and Stoltz}]{Stoltz13}
\bibinfo{author}{Gassiat, E.}, \bibinfo{author}{Pollard, D.},
  \bibinfo{author}{Stoltz, G.}, \bibinfo{year}{2013}.
\newblock \bibinfo{title}{{Revisiting the van Trees inequality in the spirit of
  Hajek and Le Cam}} \URLprefix
  \url{https://stoltz.perso.math.cnrs.fr/Publications/vanTrees-LeCam-TempWeb.pdf}.
\bibitem[{Gill and Levit(1995)}]{Gill}
\bibinfo{author}{Gill, R.D.}, \bibinfo{author}{Levit, B.Y.},
  \bibinfo{year}{1995}.
\newblock \bibinfo{title}{{Applications of the van Trees inequality: a Bayesian
  Cramér-Rao bound}}.
\newblock \bibinfo{journal}{Bernoulli} \bibinfo{volume}{1}, \bibinfo{pages}{59
  -- 79}.
\bibitem[{Giné and Nickl(2015)}]{gine_nickl_2015}
\bibinfo{author}{Giné, E.}, \bibinfo{author}{Nickl, R.}, \bibinfo{year}{2015}.
\newblock \bibinfo{title}{Mathematical Foundations of Infinite-Dimensional
  Statistical Models}.
\newblock Cambridge Series in Statistical and Probabilistic Mathematics,
  \bibinfo{publisher}{Cambridge University Press}.
\newblock \DOIprefix\doi{10.1017/CBO9781107337862}.
\bibitem[{Ibragimov and Khasminskii(1981)}]{Ibragimov81}
\bibinfo{author}{Ibragimov, I.A.}, \bibinfo{author}{Khasminskii, R.Z.},
  \bibinfo{year}{1981}.
\newblock \bibinfo{title}{Statistical estimation : asymptotic theory / I. A.
  Ibragimov, R. Z. Hasminskii ; translated by Samuel Kotz}.
\newblock \bibinfo{publisher}{Springer-Verlag New York}.
\bibitem[{Jedra and Proutiere(2019)}]{Jedra19}
\bibinfo{author}{Jedra, Y.}, \bibinfo{author}{Proutiere, A.},
  \bibinfo{year}{2019}.
\newblock \bibinfo{title}{Sample complexity lower bounds for linear system
  identification}, in: \bibinfo{booktitle}{2019 IEEE 58th Conference on
  Decision and Control (CDC)}, pp. \bibinfo{pages}{2676--2681}.
\bibitem[{Lai and Wei(1982a)}]{LAI1982346}
\bibinfo{author}{Lai, T.}, \bibinfo{author}{Wei, C.}, \bibinfo{year}{1982}a.
\newblock \bibinfo{title}{Asymptotic properties of projections with
  applications to stochastic regression problems}.
\newblock \bibinfo{journal}{Journal of Multivariate Analysis}
  \bibinfo{volume}{12}, \bibinfo{pages}{346--370}.
\bibitem[{Lai and Wei(1983)}]{LAI19831}
\bibinfo{author}{Lai, T.}, \bibinfo{author}{Wei, C.}, \bibinfo{year}{1983}.
\newblock \bibinfo{title}{Asymptotic properties of general autoregressive
  models and strong consistency of least-squares estimates of their
  parameters}.
\newblock \bibinfo{journal}{Journal of Multivariate Analysis}
  \bibinfo{volume}{13}, \bibinfo{pages}{1--23}.
\bibitem[{Lai and Wei(1982b)}]{10.1214/aos/1176345697}
\bibinfo{author}{Lai, T.L.}, \bibinfo{author}{Wei, C.Z.},
  \bibinfo{year}{1982}b.
\newblock \bibinfo{title}{{Least Squares Estimates in Stochastic Regression
  Models with Applications to Identification and Control of Dynamic Systems}}.
\newblock \bibinfo{journal}{The Annals of Statistics} \bibinfo{volume}{10},
  \bibinfo{pages}{154 -- 166}.
\bibitem[{Lang(1999)}]{Lang1999}
\bibinfo{author}{Lang, S.}, \bibinfo{year}{1999}.
\newblock \bibinfo{title}{Differential Calculus}. \bibinfo{publisher}{Springer
  New York}, \bibinfo{address}{New York, NY}.
\bibitem[{Rudelson and Vershynin(2013)}]{rudelson2013}
\bibinfo{author}{Rudelson, M.}, \bibinfo{author}{Vershynin, R.},
  \bibinfo{year}{2013}.
\newblock \bibinfo{title}{Hanson-wright inequality and sub-gaussian
  concentration}.
\newblock \bibinfo{journal}{Electron. Commun. Probab.} \bibinfo{volume}{18},
  \bibinfo{pages}{1--9}.
\bibitem[{Sarkar and Rakhlin(2019)}]{pmlr-v97-sarkar19a}
\bibinfo{author}{Sarkar, T.}, \bibinfo{author}{Rakhlin, A.},
  \bibinfo{year}{2019}.
\newblock \bibinfo{title}{Near optimal finite time identification of arbitrary
  linear dynamical systems}, in: \bibinfo{editor}{Chaudhuri, K.},
  \bibinfo{editor}{Salakhutdinov, R.} (Eds.), \bibinfo{booktitle}{Proceedings
  of the 36th International Conference on Machine Learning},
  \bibinfo{publisher}{PMLR}. pp. \bibinfo{pages}{5610--5618}.
\bibitem[{{Shirani Faradonbeh} et~al.(2018){Shirani Faradonbeh}, Tewari and
  Michailidis}]{SHIRANIFARADONBEH2018342}
\bibinfo{author}{{Shirani Faradonbeh}, M.K.}, \bibinfo{author}{Tewari, A.},
  \bibinfo{author}{Michailidis, G.}, \bibinfo{year}{2018}.
\newblock \bibinfo{title}{Finite time identification in unstable linear
  systems}.
\newblock \bibinfo{journal}{Automatica} \bibinfo{volume}{96},
  \bibinfo{pages}{342--353}.
\bibitem[{Simchowitz et~al.(2018)Simchowitz, Mania, Tu, Jordan and
  Recht}]{pmlr-v75-simchowitz18a}
\bibinfo{author}{Simchowitz, M.}, \bibinfo{author}{Mania, H.},
  \bibinfo{author}{Tu, S.}, \bibinfo{author}{Jordan, M.},
  \bibinfo{author}{Recht, B.}, \bibinfo{year}{2018}.
\newblock \bibinfo{title}{Learning without mixing: Towards a sharp analysis of
  linear system identification}, in: \bibinfo{booktitle}{Proceedings of the
  31st Conference On Learning Theory}, pp. \bibinfo{pages}{439--473}.
\bibitem[{Talagrand(2014)}]{talagrand2006generic}
\bibinfo{author}{Talagrand, M.}, \bibinfo{year}{2014}.
\newblock \bibinfo{title}{Upper and Lower Bounds for Stochastic Processes:
  Modern Methods and Classical Problems}.
\newblock \bibinfo{publisher}{Springer, Berlin, Heidelberg}.
\bibitem[{Vershynin(2018)}]{vershynin_2018}
\bibinfo{author}{Vershynin, R.}, \bibinfo{year}{2018}.
\newblock \bibinfo{title}{High-Dimensional Probability: An Introduction with
  Applications in Data Science}.
\newblock Cambridge Series in Statistical and Probabilistic Mathematics,
  \bibinfo{publisher}{Cambridge University Press}.

\end{thebibliography}
\end{document}